\documentclass[nohyperref]{article}

\usepackage{microtype}
\usepackage{subfigure}
\usepackage[linesnumbered,boxed,titlenumbered,ruled,nofillcomment,noend]{algorithm2e}
\usepackage[mathscr]{eucal}
\usepackage{bm}
\usepackage{amsthm}
\usepackage{amssymb}
\usepackage{mathtools} 
\usepackage{graphicx}
\usepackage{xr-hyper}
\usepackage{hyperref}
\usepackage[dvipsnames]{xcolor}
\usepackage[utf8]{inputenc}
\usepackage{booktabs}
\usepackage{siunitx} 
\usepackage{pifont}  
\usepackage{enumerate} 
\usepackage{stmaryrd} 
\usepackage{scalerel} 
\usepackage{quoting}
\usepackage[T1]{fontenc}
\usepackage{lmodern}

\hypersetup{
	colorlinks=true,
	citecolor=MidnightBlue,
	urlcolor=MidnightBlue,
	linkcolor=MidnightBlue,
}

\SetCommentSty{mycommfont}

\theoremstyle{plain}
\newtheorem{theorem}{Theorem}

\newtheorem{corr}[theorem]{Corollary}
\newtheorem{lemma}[theorem]{Lemma}

\theoremstyle{definition}
\newtheorem{remark}[theorem]{Remark}

\newtheorem{definition}[theorem]{Definition}

\newtheorem{example}[theorem]{Example}

\newcommand{\Ce}{\boldsymbol{\mathscr{C}}}

\newcommand{\Ge}{\boldsymbol{\mathscr{G}}}

\newcommand{\Xe}{\boldsymbol{\mathscr{X}}}

\newcommand{\F}{\textup{F}}

\newcommand{\zerobf}[0]{\bm{0}}
\newcommand{\onebf}[0]{\bm{1}}
\newcommand{\Omegabf}[0]{\bm{\Omega}}
\newcommand{\Gammabf}[0]{\bm{\Gamma}}

\newcommand{\Sigmabf}[0]{\bm{\Sigma}}

\newcommand{\Lambdabf}[0]{\bm{\Lambda}}
\newcommand{\Psibf}[0]{\bm{\Psi}}
\newcommand{\Phibf}[0]{\bm{\Phi}}

\newcommand{\Abf}[0]{\bm{A}}
\newcommand{\Bbf}[0]{\bm{B}}
\newcommand{\Cbf}[0]{\bm{C}}

\newcommand{\Gbf}[0]{\bm{G}}
\newcommand{\Hbf}[0]{\bm{H}}
\newcommand{\Ibf}[0]{\bm{I}}

\newcommand{\Mbf}[0]{\bm{M}}
\newcommand{\Nbf}[0]{\bm{N}}

\newcommand{\Qbf}[0]{\bm{Q}}

\newcommand{\Sbf}[0]{\bm{S}}
\newcommand{\Tbf}[0]{\bm{T}}
\newcommand{\Ubf}[0]{\bm{U}}
\newcommand{\Vbf}[0]{\bm{V}}
\newcommand{\Wbf}[0]{\bm{W}}
\newcommand{\Xbf}[0]{\bm{X}}
\newcommand{\Ybf}[0]{\bm{Y}}

\newcommand{\abf}[0]{\bm{a}}

\newcommand{\ebf}[0]{\bm{e}}

\newcommand{\gbf}[0]{\bm{g}}

\newcommand{\pbf}[0]{\bm{p}}
\newcommand{\qbf}[0]{\bm{q}}

\newcommand{\vbf}[0]{\bm{v}}
\newcommand{\wbf}[0]{\bm{w}}
\newcommand{\xbf}[0]{\bm{x}}
\newcommand{\ybf}[0]{\bm{y}}

\newcommand\Eb{\mathbb{E}}

\newcommand\Pb{\mathbb{P}}

\newcommand\Rb{\mathbb{R}}

\newcommand{\Dc}[0]{\mathcal{D}}

\newcommand{\Ic}{\mathcal{I}}

\DeclareMathOperator*{\argmin}{arg\,min}
\DeclareMathOperator*{\CP}{\operatorname{CP}}
\DeclareMathOperator*{\TR}{\operatorname{TR}}
\DeclareMathOperator*{\Ind}{\operatorname{Ind}}
\newcommand{\defeq}{\stackrel{\text{\tiny \textnormal{def}}}{=}}

\newcommand{\rank}[0]{\operatorname{rank}}
\newcommand{\noiter}{\textup{\#iter}}

\newcommand{\nnz}{\operatorname{nnz}}
\newcommand*{\OPT}{\textup{OPT}}

\newcommand{\trace}[0]{\operatorname{trace}}

\DeclareMathOperator*{\kr}{\bigodot}
\DeclareMathOperator*{\startimes}{\scalerel*{\circledast}{\sum}}

\usepackage{microtype}
\usepackage{graphicx}
\usepackage{subfigure}
\usepackage{booktabs}
\usepackage{hyperref}

\usepackage[accepted]{icml2022}
\usepackage{amsmath}
\usepackage{amssymb}
\usepackage{mathtools}
\usepackage{amsthm}
\usepackage[capitalize,noabbrev]{cleveref}
\usepackage[textsize=tiny]{todonotes}

\icmltitlerunning{More Efficient Sampling for Tensor Decomposition With Worst-Case Guarantees}

\begin{document}

\twocolumn[
\icmltitle{\makebox[\textwidth][s]{More Efficient Sampling for Tensor Decomposition With Worst-Case Guarantees}}

\icmlsetsymbol{equal}{*}

\begin{icmlauthorlist}
\icmlauthor{Osman Asif Malik}{lbl}

\end{icmlauthorlist}

\icmlaffiliation{lbl}{Applied Mathematics \& Computational Research Division, Lawrence Berkeley National Laboratory, Berkeley, USA}
\icmlcorrespondingauthor{Osman Asif Malik}{oamalik@lbl.gov}
\icmlkeywords{Tensor Decomposition, Randomized Algorithms, Sampling, Worst-Case Guarantees, Feature Extraction}

\vskip 0.3in
]

\printAffiliationsAndNotice{}

\begin{abstract}
Recent papers have developed alternating least squares (ALS) methods for CP and tensor ring decomposition with a per-iteration cost which is sublinear in the number of input tensor entries for low-rank decomposition.
However, the per-iteration cost of these methods still has an exponential dependence on the number of tensor modes when parameters are chosen to achieve certain worst-case guarantees.
In this paper, we propose sampling-based ALS methods for the CP and tensor ring decompositions whose cost does not have this exponential dependence, thereby significantly improving on the previous state-of-the-art. 
We provide a detailed theoretical analysis and also apply the methods in a feature extraction experiment.
\end{abstract}

\section{Introduction} \label{sec:intro}

Tensor decomposition has recently emerged as an important tool in machine learning and data mining \citep{papalexakis2016TensorsData, cichocki2016TensorNetworks, cichocki2017TensorNetworks, ji2019SurveyTensor}.
Examples of applications include parameter reduction in neural networks \citep{novikov2015TensorizingNeural, garipov2016UltimateTensorization, yang2017TensortrainRecurrent, yu2017LongtermForecasting, ye2018LearningCompact}, understanding expressiveness of deep neural networks \citep{cohen2016ExpressivePower, khrulkov2018expressive}, supervised learning \citep{stoudenmire2017SupervisedLearning, novikov2016ExponentialMachines}, filter learning \citep{hazan2005SparseImage, rigamonti2013LearningSeparable}, image factor analysis and recognition \citep{vasilescu2002MultilinearAnalysis, liu2019MachineLearning}, multimodal feature fusion \citep{hou2019DeepMultimodal}, natural language processing \citep{lei2014LowrankTensors}, feature extraction \citep{bengua2015OptimalFeature}, and tensor completion \citep{wang2017EfficientLow}.
Due to their multidimensional nature, tensors are inherently plagued by the curse of dimensionality.
Indeed, simply storing an $N$-way tensor with each dimension equal to $I$ requires $I^N$ numbers.
Tensors are also fundamentally more difficult to decompose than matrices \citep{hillar2013MostTensor}.
Many tensor decompositions correspond to difficult non-convex optimization problems.
A popular approach for tackling these optimization problems is to use alternating least squares (ALS).
While ALS works well for smaller tensors, the per-iteration cost for an $N$-way tensor of size $I \times \cdots \times I$ is $\Omega(I^N)$ since each iteration requires solving a number of least squares problems with the data tensor entries as the dependent variables.
To address this issue, several recent works have developed sampling-based ALS methods for the CP decomposition \citep{cheng2016SPALSFast, larsen2020PracticalLeverageBased} and tensor ring decomposition \citep{malik2021SamplingBasedMethod}.
When the target rank is small enough, they have a per-iteration cost which is \emph{sublinear} in the number of input tensor entries while still retaining approximation guarantees for each least squares solve with high probability.
However, the cost of these methods still has an exponential dependence on $N$: $\Omega(R^{N+1})$ for the CP decomposition and $\Omega(R^{2N+2})$ for the tensor ring decomposition, where $R$ is the relevant notion of rank.
Unlike matrix rank, both the CP and tensor ring ranks of a tensor can exceed the mode dimension $I$, in which case the previous methods would no longer have sublinear per-iteration cost.
This leads us to the following question: 
\vspace{-6pt}
\begin{quote}
	Can we construct ALS algorithms for tensor decomposition with a per-iteration cost \emph{which does not depend exponentially on $N$} and which has guarantees for each least squares solve?
\end{quote}
\vspace{-6pt}
In this paper, we show that this is indeed possible for both the CP and tensor ring%
\footnote{
	Our results are also relevant for the popular tensor train decomposition \citep{oseledets2010ApproximationBackslash, oseledets2011TensortrainDecomposition} since it is a special case of the tensor ring decomposition.
}
decompositions with high probability relative error guarantees.
Like the previous works mentioned above, we also use approximate leverage score sampling.
Unlike those previous works which use quite coarse approximations to the leverage scores, we are able to sample from a distribution which is much closer to the exact one. 
We do this by using ideas for fast leverage score estimation from \citet{drineas2012FastApproximation} combined with the recently developed recursive sketch by \citet{ahle2020ObliviousSketching}.
We also design sampling schemes for both the CP and tensor ring decompositions which allow us to avoid computing the whole sampling distribution which otherwise would cost $\Omega(I^{N-1})$.
We provide a detailed theoretical analysis and run various experiments including one on feature extraction.

When theoretical guarantees are required our methods scale much better with tensor order than the previous methods.
However, these benefits only occur when the tensor order is sufficiently high.
We therefore expect our methods to provide benefits mainly when decomposing higher order tensors.

\section{Related Work} \label{sec:related-work}

\paragraph{CP Decomposition}
\citet{cheng2016SPALSFast} propose SPALS, the first ALS algorithm for CP decomposition with a per-iteration cost sublinear in the number of input tensor entries.
They use leverage scores sampling to speed up computation of the matricized-tensor-times-Khatri--Rao product, a key kernel which arises in the ALS algorithm for CP decomposition. 
\citet{larsen2020PracticalLeverageBased} propose CP-ARLS-LEV which uses leverage score sampling to reduce the size of the least squares problems in the ALS algorithm for CP decomposition.
In addition to several practical algorithmic improvements, their relative error guarantees improve on the weaker additive error guarantees provided by \citet{cheng2016SPALSFast}.
Other papers that develop randomized algorithms for the CP decomposition include \citet{wang2015FastGuaranteed}, \citet{battaglino2018PracticalRandomized}, \citet{yang2018ParaSketchParallel} and \citet{aggour2020AdaptiveSketching}.
There are also works that use both conventional and stochastic optimization approaches \citep{sorber2012UnconstrainedOptimization, sorber2013OptimizationBasedAlgorithms, kolda2020StochasticGradients}.

\paragraph{Tensor Ring Decomposition}
\citet{yuan2019RandomizedTensor} develop a randomized method for the tensor ring decomposition which first compresses the input tensor by applying Gaussian sketches to each mode.
The compressed tensor is then decomposed using standard deterministic decomposition algorithms.
This decomposition is then combined with the sketches to get a decomposition of the original tensor.
\citet{ahmadi-asl2020RandomizedAlgorithms} develop several randomized variants of the deterministic TR-SVD algorithm by replacing the SVDs with their randomized counterpart.
\citet{malik2021SamplingBasedMethod} propose TR-ALS-Sampled which is an ALS algorithm with a per-iteration cost sublinear in the number of input tensor entries.
It uses leverage score sampling to reduce the size of the least squares problems in the standard ALS algorithm.
Other works that develop randomized methods for tensor ring decomposition include \citet{espig2012NoteTensor} and \citet{khoo2019EfficientConstruction}. 

\paragraph{Other Works on Tensor Decomposition}
Papers that develop randomized methods for other tensor decompositions include the works by
\citet{drineas2007RandomizedAlgorithm}, \citet{tsourakakis2010MachFast}, \citet{dacosta2016RandomizedMethods}, \citet{malik2018LowRankTucker}, \citet{sun2020LowrankTucker},  \citet{minster2020RandomizedAlgorithms} and \citet{fahrbach2021FastLowRank} for the Tucker decomposition; 
\citet{biagioni2015RandomizedInterpolative} and \citet{malik2020FastRandomized} for the tensor interpolative decomposition;
\citet{zhang2018RandomizedTensor} and \citet{tarzanagh2018FastRandomized} for t-product-based decompositions;
and \citet{huber2017RandomizedTensor} and \citet{che2019RandomizedAlgorithms} for the tensor train decomposition.
Papers that use skeleton approximation and other sampling-based techniques include those by \citet{mahoney2008TensorCURDecompositions}, \citet{oseledets2008TuckerDimensionality}, \citet{oseledets2010TTcrossApproximation}, \citet{caiafa2010GeneralizingColumn} and \citet{friedland2011FastLow}.

\paragraph{Sketching and Sampling}
A large body of research has been generated over the last two decades focusing on sketching and sampling techniques in numerical linear algebra; see, e.g., the review papers by \citet{halko2011FindingStructure}, \citet{mahoney2011RandomizedAlgorithms}, \citet{woodruff2014SketchingTool} and \citet{martinsson2020RandomizedNumericala}.
Of particular relevance to our work are those papers that develop sketching and sampling techniques for efficient application to matrices whose columns have Kronecker product structure.
These include sketches with particular row structure \citep{biagioni2015RandomizedInterpolative, sun2018TensorRandom, rakhshan2020TensorizedRandom, rakhshan2021RademacherRandom, iwen2021LowerMemory},
the Kronecker fast Johnson--Lindenstrauss transform \citep{battaglino2018PracticalRandomized, jin2020FasterJohnsonLindenstrauss, malik2020GuaranteesKronecker, bamberger2021JohnsonLindenstraussEmbeddings},
TensorSketch \citep{pagh2013CompressedMatrix, pham2013FastScalable, avron2014SubspaceEmbeddings, diao2018SketchingKronecker}, 
sampling-based sketches \citep{cheng2016SPALSFast, diao2019OptimalSketching, larsen2020PracticalLeverageBased, fahrbach2021FastLowRank},
and recursive sketches \citep{ahle2020ObliviousSketching}.

\section{Preliminaries} \label{sec:preliminaries}

By tensor, we mean a multidimensional array containing real numbers.
We will refer to a tensor with $N$ indices as an $N$-way or mode-$N$ tensor.
We use bold Euler script letters (e.g., $\Xe$) for tensors with three or more modes, bold uppercase letters (e.g., $\Xbf$) for matrices, bold lowercase letters (e.g., $\xbf$) for vectors, and lowercase regular letters (e.g., $x$) for scalars.
We indicate specific entries of objects with parentheses.
For example, $\Xe(i,j,k)$ is the entry on position $(i,j,k)$ in $\Xe$, and $\xbf(i)$ is the $i$th entry in $\xbf$.
A colon denotes all elements along a certain mode.
For example, $\Xe(:, k, :)$ is the $k$th lateral slice of $\Xe$, and $\Xbf(i, :)$ is the $i$th row of $\Xbf$.
We will sometimes use superscripts in parentheses to denote a sequence of objects (e.g., $\Abf^{(1)}, \ldots, \Abf^{(N)}$).
For a positive integer $n$, we use the notation $[n] \defeq \{1,\ldots,n\}$.
We use $\otimes$ and $\odot$ to denote the Kronecker and Khatri--Rao products, respectively (defined in Section~\ref{sec:notation}).
By \emph{compact SVD}, we mean an SVD $\Abf = \Ubf \Sigmabf \Vbf^\top$ where $\Sigmabf \in \Rb^{\rank(\Abf) \times \rank(\Abf)}$ and $\Ubf, \Vbf$ have $\rank(\Abf)$ columns.
The $i$th canonical basis vector is denoted by $\ebf_i$.
We denote the indicator of a random event $A$ by $\Ind\{A\}$, which is 1 if $A$ occurs and 0 otherwise.
For indices $i_1 \in [I_1], \ldots, i_N \in [I_N]$, the notation $\overline{i_1 \cdots i_N} \defeq 1 + \sum_{n=1}^N (i_n-1) \prod_{j=1}^{n-1} I_j$ will be useful when working with unfolded tensors.
See Section~\ref{sec:notation} for definitions of the asymptotic notation we use.

\begin{definition} \label{def:unfolding}
	The \emph{classical mode-$n$ unfolding} of $\Xe$ is the matrix $\Xbf_{(n)} \in \Rb^{I_n \times \Pi_{j \neq n} I_j}$ defined elementwise via
	\begin{equation}
		\Xbf_{(n)} (i_n, \overline{i_1 \cdots i_{n-1} i_{n+1} \cdots i_N}) \defeq \Xe(i_1, \ldots, i_N).
	\end{equation}
	The \emph{mode-$n$ unfolding} of $\Xe$ is the matrix $\Xbf_{[n]} \in \Rb^{I_n \times \Pi_{j \neq n} I_j}$ defined elementwise via 
	\begin{equation}
		\Xbf_{[n]} (i_n, \overline{i_{n+1} \cdots i_N i_1 \cdots i_{n-1}} ) \defeq \Xe(i_1, \ldots, i_N).
	\end{equation}
\end{definition}

\subsection{Tensor Decomposition}

We first introduce the CP decomposition.
Consider an $N$-way tensor $\Xe \in \Rb^{I_1 \times \cdots \times I_N}$.
A rank-$R$ \emph{CP decomposition} of $\Xe$ is of the form
\begin{equation} \label{eq:cp-decomposition}
	\Xe(i_1, \ldots, i_N) = \sum_{r=1}^R \prod_{j=1}^N \Abf^{(j)}(i_j, r),
\end{equation}
where each $\Abf^{(j)} \in \Rb^{I_j \times R}$ is called a \emph{factor matrix}. 
We use $\CP(\Abf^{(1)}, \ldots, \Abf^{(N)})$ to denote the tensor in \eqref{eq:cp-decomposition}.
The problem of computing a rank-$R$ CP decomposition of a data tensor $\Xe$ can be formulated as 
\begin{equation} \label{eq:cp-optimization-problem}
	\argmin_{\Abf^{(1)}, \ldots, \Abf^{(N)}} \| \CP(\Abf^{(1)}, \ldots, \Abf^{(N)}) - \Xe \|_\F.
\end{equation}
Unfortunately, this problem is non-convex and difficult to solve exactly.
ALS is the ``workhorse'' algorithm for solving this problem approximately \citep{kolda2009TensorDecompositions}.
With ALS, we consider the objective in \eqref{eq:cp-optimization-problem}, but only solve with respect to one of the factor matrices at a time while keeping the others fixed:
\begin{equation} \label{eq:cp-optimization-An}
	\argmin_{\Abf^{(n)}} \| \CP(\Abf^{(1)}, \ldots, \Abf^{(N)}) - \Xe \|_\F.
\end{equation}

The problem in \eqref{eq:cp-optimization-An} can be rewritten as the linear least squares problem
\begin{equation} \label{eq:cp-optimization-An-matrix}
	\argmin_{\Abf^{(n)}} \| \Abf^{\neq n} \Abf^{(n) \top} - \Xbf_{(n)}^\top \|_\F,
\end{equation}
where $\Abf^{\neq n} \in \Rb^{(\Pi_{j \neq n} I_j) \times R}$ is defined as 
\begin{equation} \label{eq:cp-design-matrix}
	\Abf^{\neq n} \defeq \Abf^{(N)} \odot \cdots \odot \Abf^{(n+1)} \odot \Abf^{(n-1)} \odot \cdots \odot \Abf^{(1)}.
\end{equation}
By repeatedly updating each factor matrix one at a time via \eqref{eq:cp-optimization-An-matrix}, we get the standard CP-ALS algorithm outlined in Algorithm~\ref{alg:CP-ALS}.
For further details on the CP decomposition, see \citet{kolda2009TensorDecompositions}. 

\begin{algorithm}
	\caption{CP-ALS}
	\label{alg:CP-ALS}
	\DontPrintSemicolon
	\KwIn{$\Xe \in \Rb^{I_1 \times \cdots \times I_N}$, rank $R$}
	\KwOut{Factor matrices $\Abf^{(1)}, \ldots, \Abf^{(N)}$}
	Initialize factor matrices $\Abf^{(2)}, \ldots, \Abf^{(N)}$ \label{line:cp-als:initialize-cores}\;
	
	\While{termination criteria not met}{
		\For{$n = 1,\ldots, N$}{
			$\Abf^{(n)} = \argmin_{\Abf} \| \Abf^{\neq n} \Abf^\top - \Xbf_{(n)}^\top \|_\F$ \label{line:cp-als:ls}\;
		}	
	}
	\Return{$\Abf^{(1)}, \ldots, \Abf^{(N)}$}\;
\end{algorithm}

Next, we introduce the tensor ring decomposition.
For $n \in [N]$, let $\Ge^{(n)} \in \Rb^{R_{n-1} \times I_n \times R_n}$ be 3-way tensors with $R_0 = R_N$. 
A rank-$(R_1, \ldots, R_N)$ \emph{tensor ring decomposition} of $\Xe$ is of the form
\begin{equation} \label{eq:tr-decomposition}
	\Xe(i_1, \ldots, i_N) = \sum_{r_1, \ldots, r_N} \prod_{n=1}^N \Ge^{(n)}(r_{n-1}, i_n, r_n),
\end{equation}
where each $r_n$ in the sum goes from $1$ to $R_n$ and $r_0 = r_N$, and each $\Ge^{(n)}$ is called a \emph{core tensor}. 
We use $\TR(\Ge^{(1)},\ldots,\Ge^{(N)})$ to denote the tensor in \eqref{eq:tr-decomposition}.
Finding the best possible rank-$(R_1, \ldots, R_N)$ tensor ring decomposition of a tensor $\Xe$ is difficult.
With an ALS approach we can update a single core tensor at a time by solving the following problem:
\begin{equation} \label{eq:tr-optimization-Gn}
	\argmin_{\Ge^{(n)}} \| \TR(\Ge^{(1)}, \ldots, \Ge^{(N)}) - \Xe \|_\F.
\end{equation}
To reformulate this problem into a linear least squares problem we will need the following definition.
\begin{definition} \label{def:subchain}
	By merging all cores except the $n$th, we get a \emph{subchain tensor} $\Ge^{\neq n} \in \Rb^{R_n \times (\Pi_{j \neq n} I_j) \times R_{n-1}}$ defined elementwise via
	\begin{equation}
		\begin{aligned}
			&\Ge^{\neq n}(r_n, \overline{i_{n+1} \ldots i_N i_1 \ldots i_{n-1}}, r_{n-1}) \\
			&\defeq \sum_{\substack{r_1, \ldots, r_{n-2}\\r_{n+1}, \ldots, r_N}} \prod_{\substack{j = 1\\j \neq n}}^N \Ge^{(j)}(r_{j-1}, i_j, r_j).
		\end{aligned}
	\end{equation}
\end{definition}
The problem in \eqref{eq:tr-optimization-Gn} can now be written as the linear least squares problem
\begin{equation} \label{eq:tr-optimization-Gn-matrix}
	\Ge^{(n)} = \argmin_{\Ge} \| \Gbf_{[2]}^{\neq n} \Gbf_{(2)}^\top - \Xbf_{[n]}^{\top} \|_\F.
\end{equation}
By repeatedly updating each core tensor one at a time via \eqref{eq:tr-optimization-Gn-matrix}, we get the standard TR-ALS algorithm outlined in Algorithm~\ref{alg:TR-ALS}.
For further details on the tensor ring decomposition, see \citet{zhao2016TensorRing}.
\begin{algorithm}
	\caption{TR-ALS}
	\label{alg:TR-ALS}
	\DontPrintSemicolon
	\KwIn{$\Xe \in \Rb^{I_1 \times \cdots \times I_N}$, ranks $R_1,\ldots,R_N$}
	\KwOut{Core tensors $\Ge^{(1)}, \ldots, \Ge^{(N)}$}
	Initialize core tensors $\Ge^{(2)}, \ldots, \Ge^{(N)}$ \label{line:tr-als:initialize-cores}\;
	
	\While{termination criteria not met}{
		\For{$n = 1,\ldots, N$}{
			$\Ge^{(n)} = \argmin_{\Ge} \| \Gbf_{[2]}^{\neq n} \Gbf_{(2)}^\top - \Xbf_{[n]}^{\top} \|_\F$ \label{line:tr-als:ls}\;
		}	
	}
	\Return{$\Ge^{(1)}, \ldots, \Ge^{(N)}$}\;
\end{algorithm}

\subsection{Recursive Sketching} \label{sec:gaussian-and-recursive-sketches}

\citet{ahle2020ObliviousSketching} present two variants of their recursive sketch.
The first one, which we will use, combines CountSketch and TensorSketch into a single sketch which can be applied efficiently to Kronecker structured vectors.
CountSketch was first introduced by \citet{charikar2004FindingFrequent} and extended to the linear algebra setting by \citet{clarkson2017LowRankApproximation}.
Recall that a function $h : [I] \rightarrow [J]$ is said to be \emph{$k$-wise independent} if it is chosen from a family of functions such that for any $k$ distinct $i_1, \ldots, i_k \in [I]$ the values $h(i_1), \ldots, h(i_k)$ are independent random variables uniformly distributed in $[J]$ \citep{pagh2013CompressedMatrix}.
\begin{definition}
	Let $h : [I] \rightarrow [J]$ and $s : [I] \rightarrow \{-1,+1\}$ be 3- and 4-wise independent functions, respectively.
	The \emph{CountSketch} matrix $\Cbf \in \Rb^{J \times I}$ is defined elementwise via $\Cbf(j,i) \defeq s(i) \cdot \Ind\{ h(i) = j\}$.
\end{definition}

The TensorSketch was developed in a series of papers by \citet{pagh2013CompressedMatrix}, \citet{pham2013FastScalable}, \citet{avron2014SubspaceEmbeddings} and \citet{diao2018SketchingKronecker}.
\begin{definition}
	Let $h_1, h_2 : [I] \rightarrow [J]$ and $s_1, s_2 : [I] \rightarrow \{-1,+1\}$ be 3- and 4-wise independent functions, respectively.
	Define $h : [I] \times [I] \rightarrow [J]$ via
	\begin{equation} 
		h(i_1, i_2) \defeq (h_1(i_1) + h_2(i_2) \mod J) + 1.
	\end{equation}
	The \emph{degree-two TensorSketch} matrix $\Tbf \in \Rb^{J \times I^2}$ is defined elementwise via 
	\begin{equation}
		\Tbf(j,\overline{i_1 i_2}) \defeq s(i_1) s(i_2) \cdot \Ind\{ h(i_1, i_2) = j \}.
	\end{equation}
\end{definition}

We are now ready to describe the recursive sketch of \citet{ahle2020ObliviousSketching}.
It is easiest to understand it if we consider its application to Kronecker structured vectors.
Consider $\xbf = \xbf_1 \otimes \cdots \otimes \xbf_N \in \Rb^{I_1 \cdots I_N}$, where each $\xbf_n \in \Rb^{I_n}$.%
\footnote{
	\citet{ahle2020ObliviousSketching} consider the case when each $\xbf_n$ has the same length.
	We consider a slightly more general definition here since this allows us to work with tensors whose modes are of different size.
}
Suppose first that $N = 2^q$ is a power of 2.
The first step of the recursive sketch is to apply an independent CountSketch matrix $\Cbf_n \in \Rb^{J \times I_n}$ to each $\xbf_n$:
\begin{equation} \label{eq:recursive-sketch-init}
	\ybf^{(0)}_n \defeq \Cbf_n \xbf_n \in \Rb^J, \;\;\;\; n \in [N]. 
\end{equation}
The vectors $\ybf_n^{(0)}$ are then combined pairwise using independent degree-two TensorSketches $\Tbf^{(1)}_n \in \Rb^{J \times J^2}$:
\begin{equation} \label{eq:recursive-sketch-1}
	\ybf^{(1)}_n \defeq \Tbf^{(1)}_n( \ybf^{(0)}_{2n-1} \otimes \ybf^{(0)}_{2n} ), \;\;\;\; n \in [N/2].
\end{equation}
This process is then repeated: 
At each step, pairs of length-$J$ vectors are combined using independent TensorSketches of size $J \times J^2$. 
The $m$th step is
\begin{equation} \label{eq:recursive-sketch-m}
	\ybf^{(m)}_n \defeq \Tbf^{(m)}_n( \ybf^{(m-1)}_{2n-1} \otimes \ybf^{(m-1)}_{2n} ), \;\;\;\; n \in [N/2^m].
\end{equation}
When $m=q$, we are left with a single vector $\ybf^{(q)}_1 \in \Rb^J$. 
The mapping $\xbf \mapsto \ybf_1^{(q)}$, which we denote by $\Psibf_{J, (I_n)_{n=1}^{2^q}}$, is the recursive sketch.
If $N$ is not a power of 2, we choose $q \defeq \lceil \log_2 N \rceil$ and define the recursive sketch as $\Psibf_{J, (I_n)_{n=1}^N} : \xbf \mapsto \Psibf_{J, (\tilde{I}_n)_{n=1}^{2^q} }(\xbf \otimes \ebf_1^{\otimes(2^q-N)})$, where $\ebf_1$ is the first canonical basis vector of length $I_{\max} \defeq \max_{n \in [N]} I_n$, and each $\tilde{I}_n \defeq I_{n}$ for $n \leq N$ and $\tilde{I}_n \defeq I_{\max}$ if $n > N$.
We refer to $\Psibf_{J, (I_n)_{n=1}^N}$ as a $(J, (I_n)_{n=1}^N)$-recursive sketch.
It is in fact linear, and when $N=2^q$ we can write $\Psibf_{J, (I_n)_{n=1}^N}$ as a product of $q+1$ matrices:
\begin{equation} \label{eq:recursive-sketch-factorization}
	\Psibf_{J, (I_n)_{n=1}^N} = \Tbf^{(q)} \Tbf^{(q-1)} \cdots \Tbf^{(1)} \Cbf,
\end{equation}
where $\Cbf \defeq \bigotimes_{n=1}^N \Cbf_n$ is a $J^N \times \prod_{n} I_n$ matrix and $\Tbf^{(m)} \defeq \bigotimes_{n=1}^{2^{q-m}} \Tbf_n^{(m)}$ is a $J^{2^{q-m}} \times J^{2^{q-m+1}}$ matrix.
Figure~\ref{fig:recursiveSketch} illustrates the recursive sketch for $N=4$.

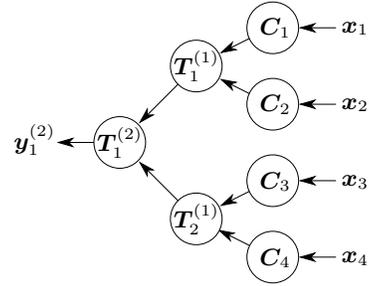
\begin{figure}[h]
	\centering  
	\includegraphics[width=.45\columnwidth]{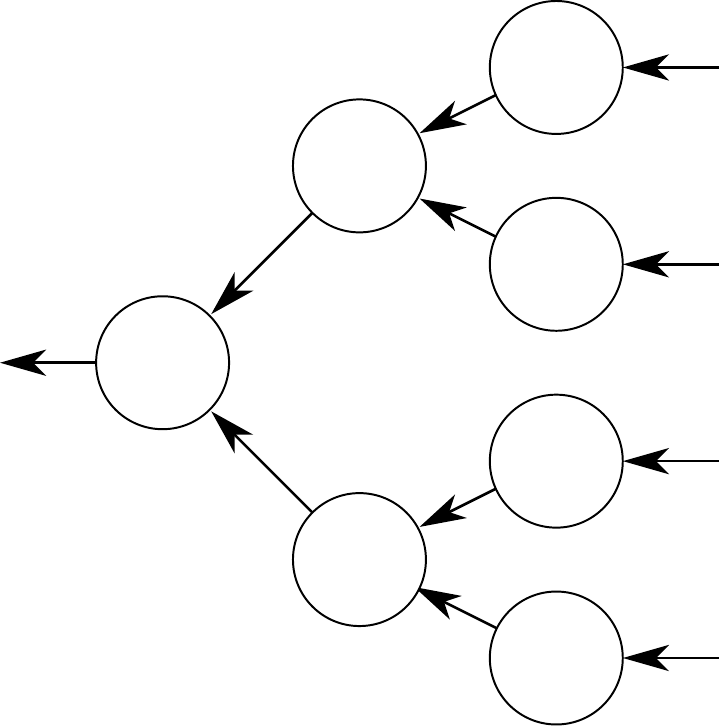}
	\caption{
		Illustration of the recursive sketch applied to a vector $\xbf_1 \otimes \xbf_2 \otimes \xbf_3 \otimes \xbf_4$.
		This is an adaption of Figure~1 in \citet{ahle2020ObliviousSketching}.
	}
	\label{fig:recursiveSketch}
	\begin{picture}(0,0)
		\put(55,150){\footnotesize $\xbf_1$}
		\put(55,121){\footnotesize $\xbf_2$}
		\put(55,92){\footnotesize $\xbf_3$}
		\put(55,63){\footnotesize $\xbf_4$}
		
		\put(24,149){\footnotesize $\Cbf_1$}
		\put(24,120){\footnotesize $\Cbf_2$}
		\put(24,91){\footnotesize $\Cbf_3$}
		\put(24,62){\footnotesize $\Cbf_4$}
		
		\put(-8.5,134){\footnotesize $\Tbf^{(1)}_1$}
		\put(-8.5,76){\footnotesize $\Tbf^{(1)}_2$}
		
		\put(-37.5,105){\footnotesize $\Tbf^{(2)}_1$}
		\put(-69,106){\footnotesize $\ybf^{(2)}_1$}
	\end{picture}
\end{figure}

The recursive sketch is a subspace embedding with high probability.
\begin{definition}
	A matrix $\Psibf \in \Rb^{J \times I}$ is called a \emph{$\gamma$-subspace embedding} for a matrix $\Abf \in \Rb^{I \times R}$ if
	\begin{equation}
		\big| \| \Psibf \Abf \xbf \|_2^2 - \|\Abf \xbf\|_2^2 \big| \leq \gamma \| \Abf \xbf \|_2^2 \;\;\;\; \text{for all } \xbf \in \Rb^R.
	\end{equation}
\end{definition}
The recursive sketch has the remarkable feature that the embedding dimension required for subspace embedding guarantees does not depend exponentially on $N$.
See Theorem~1 in \citet{ahle2020ObliviousSketching} or Theorem~\ref{thm:recursive-sketch} for a precise statement.

\subsection{Leverage Score Sampling}

Leverage score sampling is a popular technique for a variety of problems in numerical linear algebra.
For an in-depth discussion, see \citet{mahoney2011RandomizedAlgorithms} and \citet{woodruff2014SketchingTool}.

\begin{definition} \label{def:leverage-score}
	Let $\Abf \in \Rb^{I \times R}$ and suppose $\Ubf \in \Rb^{I \times \rank(\Abf)}$ contains the left singular vectors of $\Abf$. 
	The $i$th \emph{leverage score} of $\Abf$ is defined as $\ell_i(\Abf) \defeq \| \Ubf(i,:) \|_2^2$ for $i \in [I]$.
\end{definition}

\begin{definition} \label{def:leverage-score-sampling}
	Let $\qbf \in \Rb^I$ be a probability distribution and let $f : [J] \rightarrow [I]$ be a random map such that each $f(j)$ is independent and distributed according to $\qbf$.
	Define $\Sbf \in \Rb^{J \times I}$ elementwise via
	\begin{equation}
		\Sbf(j,i) \defeq \Ind\{ f(j) = i \} / \sqrt{J \qbf(f(j))}.
	\end{equation}
	We call $\Sbf$ a \emph{sampling matrix with parameters $(J,\qbf)$}, or $\Sbf \sim \Dc(J, \qbf)$ for short.
	Let $\Abf \in \Rb^{I \times R}$ be nonzero and suppose $\beta \in (0,1]$.
	Define the distribution $\pbf \in \Rb^{I}$ via $\pbf(i) \defeq \ell_i(\Abf)/\rank(\Abf)$.
	We say that $\Sbf \sim \Dc(J, \qbf)$ is a \emph{leverage score sampling matrix for $(\Abf, \beta)$} if $\qbf(i) \geq \beta \pbf(i)$ for all $i \in [I]$.
\end{definition}

For a least squares problem $\min_{\xbf} \| \Abf \xbf - \ybf \|_2$ where $\Abf \in \Rb^{I \times R}$ has many more rows than columns, we can use sampling to reduce the size of the problem to $\min_{\xbf} \| \Sbf \Abf \xbf - \Sbf \ybf \|_2$.
We would ideally like to sample according to the distribution $\pbf$ in Definition~\ref{def:leverage-score-sampling}, but this requires computing $\Ubf$ in Definition~\ref{def:leverage-score} (e.g., via the SVD or QR decomposition) which costs $O(IR^2)$.
This is the same cost as solving the full least squares problem and is therefore too expensive.
However, as shown by \citet{drineas2012FastApproximation}, the leverage scores can be accurately estimated in less time. 
Theorem~\ref{thm:leverage-score-estimation} is a variant of Lemma~9 by \citet{drineas2012FastApproximation}. 
They consider the case when $\Psibf$ is a fast Johnson--Lindenstrauss transform instead of a subspace embedding. 
\begin{theorem} \label{thm:leverage-score-estimation}
	Let $\Abf \in \Rb^{I \times R}$ where $I > R$, $\gamma \in (0,1)$ and suppose $\Psibf \Abf = \Ubf_1 \Sigmabf_1 \Vbf_1^\top$ is a compact SVD.
	Define
	\begin{equation} \label{eq:u-tilde}
		\tilde{\ell}_i(\Abf) \defeq \| \ebf_i^\top \Abf \Vbf_1 \Sigmabf_1^{-1} \|_2^2.
	\end{equation}
	Suppose that $\Psibf$ is a $\gamma$-subspace embedding for $\Abf$.
	Then
	\begin{equation}
		| \ell_i(\Abf) - \tilde{\ell}_i(\Abf) | \leq \frac{\gamma}{1-\gamma} \ell_i(\Abf) \;\;\;\; \text{for all } i \in [I].
	\end{equation}
\end{theorem}
A proof of Theorem~\ref{thm:leverage-score-estimation} appears in Section~\ref{sec:proof-leverage-score-estimation}.

\section{Efficient Sampling for Tensor Decomposition} \label{sec:efficient-sampling-for-TD}

In this section we present our proposed sampling schemes for the CP and tensor ring decompositions of a tensor $\Xe \in \Rb^{I_1 \times \cdots \times I_N}$.
We will refer to these methods as CP-ALS-ES and TR-ALS-ES, respectively, where ``ES'' is short for ``Efficient Sampling.''

\subsection{CP Decomposition} \label{sec:sampling-for-cp}

Each least squares solve on line~\ref{line:cp-als:ls} in Algorithm~\ref{alg:CP-ALS} involves all entries in $\Xe$.
To reduce the size of this problem, we sample rows according to an approximate leverage score distribution computed as in Theorem~\ref{thm:leverage-score-estimation} with $\Psibf$ chosen to be a recursive sketch.
Theorem~\ref{thm:cp-main} shows that such a sampling approach yields relative error guarantees for the CP-ALS least squares problem.
\begin{theorem} \label{thm:cp-main}
	Let $\Abf^{\neq n}$ be defined as in \eqref{eq:cp-design-matrix}.
	Define the vector $\vbf \defeq \begin{bmatrix} N, & \cdots & n+1, & n-1, & \cdots & 1\end{bmatrix}$ and suppose $\varepsilon, \delta \in (0,1)$.
	Suppose the estimates $\tilde{\ell}_i(\Abf^{\neq n})$ are computed as in Theorem~\ref{thm:leverage-score-estimation}, with $\Psibf \in \Rb^{J_1 \times \Pi_{j \neq n} I_j}$ chosen to be a $(J_1, (I_{\vbf(j)})_{j=1}^{N-1})$-recursive sketch.
	Moreover, suppose $\Sbf \in \Rb^{J_2 \times \Pi_{j \neq n} I_j}$ is a sampling matrix with parameters $(J_2, \qbf)$ where $\qbf(i) \propto \tilde{\ell}_i(\Abf^{\neq n})$.
	If 
	\begin{align}
		J_1 &\gtrsim N R^2/\delta, \label{eq:cp-J1} \\
		J_2 &\gtrsim R \, \max \big( \log(R/\delta), 1/(\varepsilon \delta)\big), \label{eq:cp-J3}
	\end{align} 
	then $\tilde{\Abf} \defeq \argmin_{\Abf} \| \Sbf \Abf^{\neq n} \Abf^\top - \Sbf \Xbf_{(n)}^\top \|_\F$ satisfies the following with probability at least $1-\delta$:
	\begin{equation} \label{eq:cp-main-relative-error}
		\| \Abf^{\neq n} \tilde{\Abf}^\top - \Xbf_{(n)}^\top \|_\F \leq (1+\varepsilon) \min_{\Abf} \| \Abf^{\neq n} \Abf^\top - \Xbf_{(n)}^\top \|_\F.
	\end{equation}
\end{theorem}

A proof of Theorem~\ref{thm:cp-main} is provided in Section~\ref{sec:proof-cp-main}.
It combines well-known results for leverage score sampling with an efficient leverage score estimation procedure.
The estimation procedure follows ideas by \citet{drineas2012FastApproximation} but uses the recursive sketch by \citet{ahle2020ObliviousSketching} instead of the fast Johnson--Lindenstrauss transform that \citeauthor{drineas2012FastApproximation} use.

The dependence on $R$ in \eqref{eq:cp-J3} is optimal in the sense that it cannot be improved when rows are sampled i.i.d. \citep{derezinski2018ReverseIterative}.
It is a significant improvement over the current state-of-the-art sampling-based ALS method by \citet{larsen2020PracticalLeverageBased} which requires $O(R^{N-1} \max(\log(R/\delta), 1/(\delta \varepsilon)))$ samples to achieve relative error guarantees.
The method by \citet{cheng2016SPALSFast} requires $O(R^N \log(I_n/\delta)/\varepsilon^2)$ samples and only achieves weaker additive error guarantees. 

In Sections~\ref{sec:cp-efficient-subspace-embedding} and \ref{sec:cp-sampling-scheme} we discuss how to compute the approximate solution $\tilde{\Abf}$ in Theorem~\ref{thm:cp-main} efficiently.
In Section~\ref{sec:cp-complexity-analysis} we compare the complexity of our method to that of other CP decomposition methods.

\subsubsection{Step 1: Computing $\Psibf \Abf^{\neq n}$} \label{sec:cp-efficient-subspace-embedding}
The columns of $\Abf^{\neq n}$ are Kronecker products, so applying the recursive sketch $\Psibf$ to $\Abf^{\neq n}$ efficiently is straightforward.
Let $q \defeq \lceil \log_2(N-1) \rceil$.
First, independent CountSketches $\Cbf_j$ with $J_1$ rows and an appropriate number of columns are applied:
\begin{equation} \label{eq:cp-recursive-init}
	\begin{aligned}
		\Ybf^{(0)}_j \defeq 
		\begin{cases}
			\Cbf_j \Abf^{(\vbf(j))} & \text{if } 1 \leq j \leq N-1, \\
			\Cbf_j \ebf_1 \onebf_{1 \times R} & \text{if } N-1 < j \leq 2^q,
		\end{cases} 
	\end{aligned}
\end{equation}
where $\vbf$ is defined as in Theorem~\ref{thm:cp-main} and $\onebf_{1 \times R}$ is a length-$R$ row vector of ones.
Then, independent TensorSketches are applied recursively:
\begin{equation} \label{eq:cp-recursive-m}
	\Ybf_j^{(m)} = \Tbf_j^{(m)} (\Ybf_{2j-1}^{(m-1)} \odot \Ybf_{2j}^{(m-1)}), \;\;\;\; j \in [2^{q-m}], 
\end{equation}
for $m = 1, \ldots, q$, where each $\Tbf_j^{(m)} \in \Rb^{J_1 \times J_1^2}$.
The final output is $\Ybf^{(q)}_1 = \Psibf \Abf^{\neq n}$.

\subsubsection{Step 2: Drawing Samples Efficiently} \label{sec:cp-sampling-scheme}
Since a row index $i \in [\prod_{j \neq n} I_j]$ of $\Abf^{\neq n}$ can be written as $i = \overline{i_1 \cdots i_{n-1} i_{n+1} \cdots i_N}$ where each $i_j \in [I_j]$, we can sample an index $i \in [\prod_{j \neq n} I_j]$ by sampling subindices $i_j \in [I_j]$ for each $j \neq n$.
By sampling the subindices in sequence one after another we avoid computing all entries in $\qbf$ which otherwise would cost $\Omega(\prod_{j \neq n} I_j)$.
We use an abbreviated notation to denote the probability of drawing subsequences of indices.
For example, $\Pb(i_1)$ denotes the probability that the first index is $i_1$, and $\Pb( (i_j)_{\substack{j \leq m, j \neq n}} )$ denotes the probability that the first $m$ indices (excluding the $n$th) are $i_1, \ldots, i_{n-1}, i_{n+1}, \ldots, i_m$.
\begin{lemma} \label{lem:cp-sampling-results}
	Let $\Psibf\Abf^{\neq n} = \Ubf_1 \Sigmabf_1 \Vbf_1^\top$ be a compact SVD and define $\Phibf \defeq \Vbf_1 \Sigmabf_1^{-1} (\Vbf_1 \Sigmabf_1^{-1})^\top$.
	The normalization constant for the distribution $\qbf$ with $\qbf(i) \propto \tilde{\ell}_i(\Abf^{\neq n})$ is 
	\begin{equation} \label{eq:cp-sampling-results-C}
		C 
		\defeq \sum_{i} \tilde{\ell}_i(\Abf^{\neq n}) 
		= \sum_{r, k} \Phibf(r, k) \cdot \prod_{j \neq n} (\Abf^{(j)\top} \Abf^{(j)})(r,k).
	\end{equation}
	The marginal probability of drawing $(i_j)_{\substack{j \leq m, j \neq n}}$ is
	\begin{equation} \label{eq:cp-sampling-results-joint-prob}
		\begin{aligned}
			&\Pb( (i_j)_{\substack{j \leq m, j \neq n}} ) = \frac{1}{C} \sum_{r, k} \Phibf(r, k) \\ 
			& \cdot \Big( \prod_{\substack{j \leq m \\ j \neq n}} \Abf^{(j)}(i_j, r) \Abf^{(j)}(i_j, k) \Big) \Big( \prod_{\substack{j > m \\ j \neq n}} \big( \Abf^{(j)\top} \Abf^{(j)} \big)(r, k) \Big).
		\end{aligned}
	\end{equation}
	In \eqref{eq:cp-sampling-results-C} and \eqref{eq:cp-sampling-results-joint-prob}, the summations are over $i \in [\prod_{j \neq n} I_j]$ and $r, k \in [R]$.
\end{lemma}

The proof of Lemma~\ref{lem:cp-sampling-results} is given in Section~\ref{sec:proof-cp-sampling-results}.
We now describe the sampling procedure by first describing how to sample the first index $i_1$ (or $i_2$, if $n = 1$), followed by all subsequent indices.

\paragraph{Sampling First Index} 
Suppose $n \neq 1$.
We compute the probability of sampling $i_1$ for all $i_1 \in [I_1]$ via \eqref{eq:cp-sampling-results-joint-prob} and sample an index $i_1 \in [I_1]$ from that distribution.
If $n = 1$, we do this for the second index $i_2$ instead.

\paragraph{Sampling Subsequent Indices}
After drawing $i_1$ (or $i_2$, if $n=1$), all subsequent indices can be drawn one at a time \emph{conditionally} on the previous indices.
Suppose we have drawn indices $(i_j)_{j < m, j \neq n}$. 
The conditional distribution of $i_m$ (or $i_{m+1}$, if $n = m$) given the previously drawn indices is then
\begin{equation} \label{eq:conditional-probability-im} 
	\Pb(i_m \mid (i_j)_{j < m, j \neq n}) = \frac{\Pb((i_j)_{j \leq m, j \neq n})}{\Pb((i_j)_{j < m, j \neq n})}.
\end{equation} 
We compute the conditional probability in \eqref{eq:conditional-probability-im} for all $i_m \in [I_m]$ via \eqref{eq:cp-sampling-results-joint-prob} and draw a sample $i_m$ from that distribution.

Once the $J_2$ samples in $[\prod_{j \neq n} I_j]$ have been drawn, the matrix $\Sbf \Abf^{\neq n}$ can be computed without forming $\Abf^{\neq n}$. 
The matrix $\Sbf \Xbf_{(n)}^{\top}$ can be computed by extracting only $J_2$ rows from $\Xbf_{(n)}^\top$.

\subsubsection{Complexity Analysis} \label{sec:cp-complexity-analysis}

If $J_1$ and  $J_2$ are chosen as in \eqref{eq:cp-J1} and \eqref{eq:cp-J3}, and if we assume that $I_j = I$ for all $j \in [N]$ and ignore $\log$ factors, then the per-iteration complexity for our method CP-ALS-ES is 
$\tilde{O}( N^2 R^3 (R + N I / \varepsilon) / \delta )$.
In Table~\ref{tab:cp-complexity-comparison}, we compare this to the complexity of other ALS-based methods for CP decomposition (see Section~\ref{sec:related-work}).
Our method is the only one that does not have an exponential per-iteration dependence on $N$.
See Section~\ref{sec:detailed-complexity} for a detailed complexity analysis.

\begin{table}[ht!]
	\centering
	\caption{
		Comparison of leading order computational cost for various CP decomposition methods.  
		We ignore $\log$ factors and assume that $I_j = I$ for all $j \in [N]$.
		$\noiter$ is the number of ALS iterations.
		SPALS has an additional upfront cost of $\nnz(\Xe)$.
		\label{tab:cp-complexity-comparison}
	}
	\begin{tabular}{ll}  
		\toprule
		Method    					& Complexity 												\\
		\midrule
		CP-ALS						& $\noiter \cdot N (N+I) I^{N-1} R$ 						\\
		SPALS 						& $\noiter \cdot N (N+I) R^{N+1} / \varepsilon^2$ 			\\	
		CP-ARLS-LEV					& $\noiter \cdot N ( R + I ) R^N / (\delta \varepsilon)$	\\
		\textbf{CP-ALS-ES (our)}    & $\noiter \cdot N^2 R^3 (R + N I / \varepsilon) / \delta$	\\
		\bottomrule
	\end{tabular}
\end{table}

\subsection{Tensor Ring Decomposition} \label{sec:sampling-for-tr}

The least squares problem for TR-ALS on line~\ref{line:tr-als:ls} in Algorithm~\ref{alg:TR-ALS} also involves all entries in $\Xe$.
We use an approach similar to that for the CP decomposition, which yields the following approximation guarantees for the TR-ALS least squares problem.
\begin{theorem} \label{thm:tr-main}
	Let $\Gbf_{[2]}^{\neq n}$ be the mode-2 unfolding of the subchain tensor $\Ge^{\neq n}$ (see Definitions~\ref{def:unfolding} and \ref{def:subchain}).
	Define the vector $\wbf \defeq \begin{bmatrix}n-1, & \cdots & 1, & N, & \cdots & n+1\end{bmatrix}$ and suppose $\varepsilon, \delta \in (0,1)$.
	Suppose the estimates $\tilde{\ell}_i(\Gbf_{[2]}^{\neq n})$ are computed as in Theorem~\ref{thm:leverage-score-estimation}, with $\Psibf \in \Rb^{J_1 \times \Pi_{j \neq n} I_j}$ chosen to be a $(J_1, (I_{\wbf(j)})_{j=1}^{N-1})$-recursive sketch.
	Moreover, suppose $\Sbf \in \Rb^{J_2 \times \Pi_{j \neq n} I_j}$ is a sampling matrix with parameters $(J_2, \qbf)$ where $\qbf(i) \propto \tilde{\ell}_i(\Gbf_{[2]}^{\neq n})$.
	If 
	\begin{align}
		J_1 &\gtrsim N (R_{n-1} R_n)^2/\delta, \label{eq:tr-J1} \\
		J_2 &\gtrsim R_{n-1} R_n \, \max \big( \log(R_{n-1} R_n/\delta), 1/(\varepsilon \delta)\big), \label{eq:tr-J3}
	\end{align} 
	then $\tilde{\Gbf} \defeq \argmin_{\Gbf} \| \Sbf \Gbf_{[2]}^{\neq n} \Gbf^\top - \Sbf \Xbf_{[n]}^\top \|_\F$ satisfies the following with probability at least $1-\delta$:
	\begin{equation} \label{eq:tr-main-relative-error}
		\| \Gbf_{[2]}^{\neq n} \tilde{\Gbf}^\top - \Xbf_{[n]}^\top\|_\F \leq (1+\varepsilon) \min_{\Gbf} \|\Gbf_{[2]}^{\neq n} \Gbf^\top - \Xbf_{[n]}^\top \|_\F.
	\end{equation}
\end{theorem}

A proof of Theorem~\ref{thm:tr-main} is provided in Section~\ref{sec:proof-tr-main}.
The proof uses similar steps as the proof of Theorem~\ref{thm:cp-main}.
The main difference is the structure and number of columns of the least squares design matrix.

Since $\Gbf_{[2]}^{\neq n}$ has $R_{n-1} R_n$ columns, the sample complexity in \eqref{eq:tr-J3} has optimal rank dependence in the sense discussed in Section~\ref{sec:sampling-for-cp}.
This is a significant improvement over the current state-of-the-art sampling-based ALS method by \citet{malik2021SamplingBasedMethod} which requires $O((\prod_{j} R_j^2) \max(\log(R_{n-1} R_n / \delta)),1/(\varepsilon\delta))$ samples to achieve relative error guarantees.

In Sections~\ref{sec:tr-efficient-subspace-embedding} and \ref{sec:tr-sampling-scheme} we discuss how to compute the approximate solution $\tilde{\Gbf}$ in Theorem~\ref{thm:tr-main} efficiently.
In Section~\ref{sec:tr-complexity-analysis} we compare the complexity of our method to that of other tensor ring methods.

\subsubsection{Step 1: Computing $\Psibf \Gbf_{[2]}^{\neq n}$} \label{sec:tr-efficient-subspace-embedding}

Although $\Gbf_{[2]}^{\neq n}$ has a more complicated structure than $\Abf^{\neq n}$, $\Psibf$ can still be applied efficiently to $\Gbf_{[2]}^{\neq n}$.
We describe a scheme for computing the column $\Psibf \Gbf_{[2]}^{\neq n}(:, \overline{r_{n-1} r_n})$ below, and give a more detailed motivation in Section~\ref{sec:proof-correctness-of-tr-sketch}.
Let $q \defeq \lceil \log_2 (N-1) \rceil$.
Define matrices $\Hbf^{(j)}$ for $j \in [2^q]$ as follows: 
Let $\Hbf^{(1)} \in \Rb^{I_{n-1} \times R_{n-2}}$ be a matrix with columns $\Hbf^{(1)}(:, k) \defeq \Gbf_{[2]}^{(n-1)}(:, \overline{r_{n-1} k})$ for $k \in [R_{n-2}]$.
Let $\Hbf^{(j)} \defeq \Gbf_{[2]}^{(\wbf(j))} \in \Rb^{I_{\wbf(j)} \times R_{\wbf(j)} R_{\wbf(j)-1}}$ for $2 \leq j \leq N-2$.
Let $\Hbf^{(N-1)} \in \Rb^{I_{n+1} \times R_{n+1}}$ be a matrix with columns $\Hbf^{(N-1)}(:, k) \defeq \Gbf_{[2]}^{(n+1)}(:, \overline{k r_n})$ for $k \in [R_{n+1}]$.
Let $\Hbf^{(j)} \defeq \ebf_1 \in \Rb^{\max_{j \neq n} I_j}$ be a column vector for $N \leq j \leq 2^q$.
Next, define
\begin{align}
	\Ybf_j^{(0)} &\defeq \Cbf_j \Hbf^{(j)}, \;\; j \in [2^q], \label{eq:Y0j}\\
	K_j^{(0)} &\defeq
	\begin{cases}
		R_{\wbf(j)} 	& \text{if } 2 \leq j \leq N-1, \\ 
		1 				& \text{if } j = 1 \text{ or } N \leq j \leq 2^q+1.
	\end{cases}
\end{align}
The TensorSketch matrices are then applied recursively as follows.
For each $m = 1, \ldots, q$, compute
\begin{equation} \label{eq:Ymj}
	\begin{aligned} 
		&\Ybf_j^{(m)}(:, \overline{k_{1} k_{3}}) = \\
		&\sum_{k_{2} \in [K_{2j}^{(m-1)}]} \Tbf_j^{(m)} \big( \Ybf^{(m-1)}_{2j-1}(:, \overline{k_{1} k_{2}}) \otimes \Ybf^{(m-1)}_{2j}(:, \overline{k_{2} k_{3}}) \big)\\
	\end{aligned}
\end{equation}
for each $k_1 \in [K^{(m-1)}_{2j-1}], k_3 \in [K^{(m-1)}_{2j+1}], j \in [2^{q-m}]$. For each $m = 1, \ldots, q$, also compute 
\begin{equation}
	K_{j}^{(m)} \defeq K_{2j-1}^{(m-1)}, \;\; j \in [2^{q-m}+1].
\end{equation}
We prove the following in Section~\ref{sec:proof-correctness-of-tr-sketch}.
\begin{lemma} \label{lem:correctness-of-tr-sketch}
	$\Ybf_1^{(q)}$ satisfies $\Ybf_1^{(q)} = \Psibf \Gbf_{[2]}^{\neq n}(:, \overline{r_{n-1} r_n})$.
\end{lemma}
The entire matrix $\Psibf \Gbf_{[2]}^{\neq n}$ can be computed by repeating the steps above for each column $\overline{r_{n-1} r_n} \in [R_{n-1} R_n]$.

\subsubsection{Step 2: Drawing Samples Efficiently} \label{sec:tr-sampling-scheme} 

The sampling approach for the tensor ring decomposition is similar to the approach for the CP decomposition which we described in Section~\ref{sec:cp-sampling-scheme}.
\begin{lemma} \label{lem:tr-sampling-results}
	Let $\Psibf \Gbf_{[2]}^{\neq n} = \Ubf_1 \Sigmabf_1 \Vbf_1^\top$ be a compact SVD and define  $\Phibf \defeq \Vbf_1 \Sigmabf_1^{-1} (\Vbf_1 \Sigmabf_1^{-1})^\top$.
	The normalization constant for the distribution $\qbf$ with $\qbf(i) \propto \tilde{\ell}_i(\Gbf_{[2]}^{\neq n})$ is 
	\begin{equation} \label{eq:tr-sampling-results-C}
		\begin{aligned}
			&C \defeq \sum_{i} \tilde{\ell}_i(\Gbf_{[2]}^{\neq n}) = \sum_{\substack{r_1,\ldots,r_N \\ k_1, \ldots, k_N}} \Phibf(\overline{r_{n-1} r_n}, \overline{k_{n-1} k_n}) \\
			&\cdot \prod_{j \neq n} \big(\Gbf_{[2]}^{(j)\top} \Gbf_{[2]}^{(j)}\big)(\overline{r_{j} r_{j-1}}, \overline{k_{j} k_{j-1}}).
		\end{aligned}	
	\end{equation}
	The marginal probability of drawing $(i_j)_{j \leq m, j \neq n}$ is
	\begin{equation} \label{eq:tr-sampling-results-joint-prob}
		\begin{aligned}
			&\Pb( (i_j)_{j \leq m, j \neq n} ) = \frac{1}{C} \sum_{\substack{r_1, \ldots, r_N \\ k_1, \ldots, k_N}} \Phibf(\overline{r_{n-1} r_n}, \overline{k_{n-1} k_n}) \\ 
			& \cdot \Big( \prod_{\substack{j \leq m \\ j \neq n}} \Gbf_{[2]}^{(j)}(i_j, \overline{r_{j} r_{j-1}}) \Gbf_{[2]}^{(j)}(i_j, \overline{k_{j} k_{j-1}}) \Big) \\
			& \cdot \Big( \prod_{\substack{j > m \\ j \neq n}} \big( \Gbf_{[2]}^{(j)\top} \Gbf_{[2]}^{(j)} \big)(\overline{r_{j} r_{j-1}}, \overline{k_{j} k_{j-1}}) \Big).
		\end{aligned}
	\end{equation}
	In \eqref{eq:tr-sampling-results-C} and \eqref{eq:tr-sampling-results-joint-prob}, the summations are over $i \in [\prod_{j \neq n} I_j]$ and $r_j, k_j \in [R_j]$ for each $j \in [N]$.
\end{lemma}

The proof of Lemma~\ref{lem:tr-sampling-results} is given in Section~\ref{sec:proof-tr-sampling-results}.
The sampling procedure itself is the same as for the CP decomposition.
The distribution $(\Pb(i_1))_{i_1=1}^{I_1}$ is computed via \eqref{eq:tr-sampling-results-joint-prob} and an index $i_1$ is sampled; if $n=1$, these computations are done for $i_2$ instead of $i_1$. 
All subsequent indices are then sampled conditionally on the previous indices.
This is done by computing the conditional distribution in \eqref{eq:conditional-probability-im} by using \eqref{eq:tr-sampling-results-joint-prob}.
The expression in \eqref{eq:tr-sampling-results-joint-prob} can be computed efficiently despite the exponential number of terms in the summation; see Remark~\ref{remark:computing-tr-probabilities} for details.

Once the $J_2$ samples in $[\prod_{j \neq n} I_j]$ have been drawn, the matrix $\Sbf \Gbf_{[2]}^{\neq n}$ can be computed without forming $\Gbf_{[2]}^{\neq n}$.
We describe this in detail in Remark~\ref{remark:computing-sketched-tr-design-matrix}. 
The matrix $\Sbf \Xbf_{[2]}^\top$ can be computed by extracting only $J_2$ rows from $\Xbf_{[2]}^\top$.

\subsubsection{Complexity Analysis} \label{sec:tr-complexity-analysis}

If $J_1$ and $J_2$ are chosen as in \eqref{eq:tr-J1} and \eqref{eq:tr-J3}, and if we assume that $R_j = R$ and $I_j = I$ for all $j \in [N]$ and ignore $\log$ factors, then the per-iteration complexity for our method TR-ALS-ES is $\tilde{O}( N^3 R^9/\delta + N^3 I R^8/(\varepsilon \delta) )$.
In Table~\ref{tab:tr-complexity-comparison}, we compare this to the complexity of several other methods for tensor ring decomposition (see Section~\ref{sec:related-work}).
rTR-ALS refers to the method by \citet{yuan2019RandomizedTensor} with each $I$ compressed to $K$ and with TR-ALS as the deterministic algorithm. 
TR-SVD-Rand refers to Algorithm~7 in \citet{ahmadi-asl2020RandomizedAlgorithms}.
Our method is the only one that does not have an explicit exponential dependence on $N$.
See Section~\ref{sec:detailed-complexity} for a detailed complexity analysis. 

\begin{table}[ht!]
	\centering
	\caption{
		Comparison of leading order computational cost for various tensor ring decomposition methods. 
		We ignore $\log$ factors and assume that $R_j = R$ and $I_j = I$ for all $j \in [N]$.
		$\noiter$ is the number of ALS iterations.
		\label{tab:tr-complexity-comparison}
	}
	\begin{tabular}{ll}  
		\toprule
		Method    						& Complexity \\
		\midrule
		TR-ALS							& $\noiter \cdot N I^N R^2$ \\
		rTR-ALS							& $N I^N K + \noiter \cdot N K^N R^2$ \\
		TR-SVD							& $I^{N+1} + I^N R^3$ \\
		TR-SVD-Rand						& $I^N R^2$ \\
		TR-ALS-Sampled						& $\noiter \cdot N I R^{2N+2} / (\varepsilon \delta)$	\\
		\textbf{TR-ALS-ES (our)}	   	& $\noiter \cdot N^3 R^8 ( R + I / \varepsilon)/\delta$	\\
		\bottomrule
	\end{tabular}
\end{table}

\section{Experiments} \label{sec:experiments}

The experiments are run in Matlab R2021b on a laptop computer with an Intel Core i7-1185G7 CPU and 32 GB of RAM.
Our code is available at \url{https://github.com/OsmanMalik/TD-ALS-ES}.
Additional experiment details are in Section~\ref{sec:additional-experiments}.

\subsection{Sampling Distribution Comparison} \label{sec:sampling-ditribution-comparison}

We first compare the sampling distributions used by our methods with those used by the previous state-of-the-art---CP-ARLS-LEV by \citet{larsen2020PracticalLeverageBased} for the CP decomposition and TR-ALS-Sampled by \citet{malik2021SamplingBasedMethod} for the tensor ring decomposition---when solving the least squares problems in \eqref{eq:cp-optimization-An} and \eqref{eq:tr-optimization-Gn}.
We run standard CP-ALS and TR-ALS on a real data tensor to get realistic factor matrices and core tensors when defining the design matrices $\Abf^{\neq n}$ and $\Gbf_{[2]}^{\neq n}$. 
We get the real data tensor $\Xe \in \Rb^{16 \times \cdots \times 16}$ by reshaping a $4096 \times 4096$ gray scale image of a tabby cat into a 6-way tensor and then appropriately permuting the modes, a process called visual data tensorization \citep{yuan2019HighorderTensor}.
We then consider the least squares problems corresponding to an update of the 6th factor matrix or core tensor.
As a performance measure, we compute the KL-divergence of the approximate distribution $\qbf$ from the exact leverage score sampling distribution $\pbf$ in Definition~\ref{def:leverage-score-sampling}.
Tables~\ref{tab:KL-div-CP} and \ref{tab:KL-div-TR} report the results for different $J_1$ and ranks.
The results show that our methods sample from a distribution much closer to the exact leverage score distribution when $J_1$ is as small as $J_1 = 1000$.
See Figures~\ref{fig:CP-rank-10-probability-comparison}--\ref{fig:TR-rank-5-probability-comparison} for a graphical comparison.

\begin{table}[ht!]
	\centering
	\caption{
		KL-divergence (lower is better) of the approximated sampling distribution from the exact one for a CP-ALS least squares problem \eqref{eq:cp-optimization-An} with target rank $R$.
		\label{tab:KL-div-CP}
	}
	\begin{tabular}{lll}  
		\toprule
		Method & $R=10$ & $R=20$ \\
		\midrule
		CP-ARLS-LEV 						& 0.2342 & 0.1853 \\
		CP-ALS-ES ($J_1 = \text{1e+4}$)		& 0.0005 & 0.0006 \\
		CP-ALS-ES ($J_1 = \text{1e+3}$) 	& 0.0151 & 0.0070 \\
		CP-ALS-ES ($J_1 = \text{1e+2}$) 	& 0.1416 & 0.2173 \\
		\bottomrule
	\end{tabular}
\end{table}

\begin{table}[ht!]
	\centering
	\caption{
		KL-divergence (lower is better) of the approximated sampling distribution from the exact one for a TR-ALS least squares problem \eqref{eq:tr-optimization-Gn-matrix} with target rank $(R,\ldots,R)$.
		\label{tab:KL-div-TR}
	}
	\begin{tabular}{lll}  
		\toprule
		Method & $R=3$ & $R=5$ \\
		\midrule
		TR-ALS-Sampled 						& 0.3087 & 0.1279 \\
		TR-ALS-ES ($J_1 = \text{1e+4}$) 	& 0.0005 & 0.0007 \\
		TR-ALS-ES ($J_1 = \text{1e+3}$) 	& 0.0076 & 0.0070 \\
		TR-ALS-ES ($J_1 = \text{1e+2}$) 	& 0.1565 & 0.1831 \\
		\bottomrule
	\end{tabular}
\end{table}

\subsection{Feature Extraction} \label{sec:feature-extraction}

Next, we run a benchmark feature extraction experiment on the COIL-100 image dataset \citep{nene1996ColumbiaObject} with a setup similar to that in \citet{zhao2016TensorRing} and \citet{malik2021SamplingBasedMethod}.
The data consists of 7200 color images of size $128 \times 128$ pixels, each belonging to one of 100 different classes.
The data is arranged into a $128 \times 128 \times 3 \times 7200$ tensor which is decomposed using either a rank-25 CP decomposition or a rank-$(5,5,5,5)$ tensor ring decomposition.
The mode-4 factor matrix or core tensor is then used as a feature matrix in a $k$-NN algorithm with $k=1$ and 10-fold cross validation.
For our CP-ALS-ES, we use $J_1 = 10000$ and $J_2 = 2000$, and for our TR-ALS-ES we use $J_1 = 10000$ and $J_2 = 1000$.
For the CP decomposition, we compare against CP-ALS in Tensor Toolbox \citep{bader2006Algorithm862, bader2021MATLABTensor}; CPD-ALS, CPD-MINF and CPD-NLS in Tensorlab \citep{vervliet2016Tensorlab}; and our own implementation of CP-ARLS-LEV.
For the tensor ring decomposition, we compare against the implementations of TR-ALS and TR-ALS-Sampled provided by \citet{malik2021SamplingBasedMethod}.
For CP-ARLS-LEV and TR-ALS-Sampled we use 2000 and 1000 samples, respectively.
All iterative methods are run for 40 iterations.

Table~\ref{tab:classification-results} shows the average decomposition time, decomposition error, and classification accuracy for the various methods over 10 repetitions of the experiment\footnote{TR-ALS was only run once due to how long it takes to run.}.
For the CP decomposition, the two randomized methods are faster than the competing methods.
Our method takes about as long to run as CP-ARLS-LEV.
This does not contradict our earlier analysis, which was a \emph{worst-case} analysis. 
Real-world datasets are typically better behaved than the worst case, which is why CP-ARLS-LEV requires no more samples than CP-ALS-ES in this example.
For the tensor ring decomposition, the randomized methods are substantially faster than the deterministic one. 
Our method is a bit slower here than TR-ALS-Sampled.
All methods achieve good classification accuracy and similar decomposition errors.

\begin{table}[ht!]
	\centering
	\caption{
		Run time, decomposition error and classification accuracy when using tensor decomposition for feature extraction.
		\label{tab:classification-results}
	}
	\begin{tabular}{lrrr}  
		\toprule
		Method    						& Time (s) & Err.\ & Acc.\ (\%) \\
		\midrule
		CP-ALS (Ten.\ Toolbox)			&  43.6 & 0.31 & 99.2 \\
		CPD-ALS (Tensorlab)				&  68.4 & 0.31 & 99.0 \\
		CPD-MINF (Tensorlab)			& 102.3 & 0.34 & 99.7 \\
		CPD-NLS (Tensorlab)				& 107.8 & 0.31 & 92.1 \\
		CP-ARLS-LEV						&  28.7 & 0.32 & 98.5 \\
		\textbf{CP-ALS-ES (our)}		&  27.9 & 0.32 & 98.3 \\
		\midrule
		TR-ALS							& 9813.7 & 0.31 & 99.3 \\
		TR-ALS-Sampled					&    9.9 & 0.33 & 98.5 \\
		\textbf{TR-ALS-ES (our)}	   	&   28.5 & 0.33 & 98.0 \\
		\bottomrule
	\end{tabular}
\end{table}

\begin{remark}
	If $k$-NN was applied directly to the uncompressed images its cost would scale linearly or worse with the number of tensor entries.
	Due to the sublinear per-iteration complexity of our proposed methods, the cost of the entire decomposition is sublinear if the number of iterations is chosen appropriately.
	While fixing the number of iterations is not guaranteed to produce a good decomposition, we expect this to work well on typical datasets.
	Once the decomposition is computed each image has a representation of much lower dimension which makes applying $k$-NN cheaper.
	This leads to a reduction in the overall classification cost.
	See Section~\ref{sec:supp-feature-extraction-experiments} for a discussion on handling new images that were not part of the initial decomposition.
\end{remark}

\subsection{Demonstration of Improved Complexity}

We construct a synthetic 10-way tensor that demonstrates the improved sampling and computational complexity of our proposed CP-ALS-ES over CP-ARLS-LEV.
It is constructed via \eqref{eq:cp-decomposition} from factor matrices $\Abf^{(n)} \in \Rb^{6 \times 4}$ for $n \in [10]$ with $\Abf^{(n)}(1,1) = 4$, each $\Abf^{(n)}(i,j)$ drawn i.i.d.\ from a Gaussian distribution for $2 \leq i,j \leq 6$, and all other entries zero.
Additionally, i.i.d.\ Gaussian noise with standard deviation $0.01$ is added to all entries of the tensor.
Both methods are run for 20 iterations with a target rank of 4 and are initialized using the randomized range finding approach proposed by \citet{larsen2020PracticalLeverageBased}, Appendix~F.
CP-ARLS-LEV requires $J=6^8 \approx 1.7\text{e+6}$ samples to get an accurate solution, taking 350 seconds.
By contrast, our CP-ALS-ES only requires a recursive sketch size of $J_1 = 1000$ and $J_2 = 50$ samples to get an accurate solution, taking only 4.8 seconds. 
Our method improves the sampling complexity and compute time by 4 and almost 2 orders of magnitude, respectively.
A similar example for the tensor ring decomposition is provided in Section~\ref{sec:demonstration-improvement-TR}.

\section{Discussion and Conclusion} \label{sec:conclusion}

We have shown that it is possible to construct ALS algorithms with guarantees for both the CP and tensor ring decompositions of an $N$-way tensor with a per-iteration cost which does not depend exponentially on $N$.
In the regime of high-dimensional tensors (i.e., with many modes), this is a substantial improvement over the previous state-of-the-art which had a per-iteration cost of $\Omega(R^{N+1})$ and $\Omega(R^{2N + 2})$ for the CP and tensor ring decompositions, respectively, where $R$ is the relevant notion of rank.

We again want to emphasize that this paper considers \emph{worst-case} guarantees.
As we saw in Section~\ref{sec:experiments}, real datasets typically behave better than the worst case.
For such datasets, CP-ARLS-LEV and TR-ALS-Sampled usually yield good results even with fewer (i.e., not exponential in $N$) number of samples than what worst-case analysis might suggest.
In such cases, those methods can be faster than the methods we propose in this paper.
Nonetheless, we believe that our methods are still useful for those cases when worst-case performance is critical.

We also want to point out that, unlike their deterministic counterparts, our methods cannot guarantee a monotonically decreasing objective value.
The other randomized methods we compare with have the same deficiency.

The sampling formulas \eqref{eq:cp-sampling-results-joint-prob} and \eqref{eq:tr-sampling-results-joint-prob} are sums of products where each feature matrix or core tensor (except the $n$th) appears twice.
This can exacerbate issues with ill-conditioned factor matrices or core tensors.
A particular concern is that catastrophic cancellation can occur, which will prevent accurate computation of the probabilities.
Addressing this issue is an interesting direction for future research.

All the experiments in this paper involve dense tensors.
However, our methods can also be applied to sparse tensors with only minor adjustments to the code.
Thorough empirical evaluation of our methods on sparse tensors is therefore another interesting direction for future research.

\section*{Acknowledgements}

We thank the anonymous reviewers for their feedback which helped improve the paper.
This work was supported by the Laboratory Directed Research and Development Program of Lawrence Berkeley National Laboratory under U.S.\ Department of Energy Contract No.\ DE-AC02-05CH11231.

\bibliography{library}
\bibliographystyle{icml2022}

%%%%%%%%%%%%%%%%%%%%%%%%%%%%%%%%%%%%%%%%%%%%%%%%%%%%%%%%%%%%%%%%%%%%%%%%%%%%%%%
%%%%%%%%%%%%%%%%%%%%%%%%%%%%%%%%%%%%%%%%%%%%%%%%%%%%%%%%%%%%%%%%%%%%%%%%%%%%%%%
% APPENDIX
%%%%%%%%%%%%%%%%%%%%%%%%%%%%%%%%%%%%%%%%%%%%%%%%%%%%%%%%%%%%%%%%%%%%%%%%%%%%%%%
%%%%%%%%%%%%%%%%%%%%%%%%%%%%%%%%%%%%%%%%%%%%%%%%%%%%%%%%%%%%%%%%%%%%%%%%%%%%%%%
\newpage
\appendix
\onecolumn

\section{Further Details on Notation} \label{sec:notation}

In this section we provide some further details on the notation used in this paper to help make the paper self contained.
The \emph{Kronecker product} of two matrices $\Abf \in \Rb^{m \times n}$ and $\Bbf \in \Rb^{k \times \ell}$ is denoted by $\Abf \otimes \Bbf \in \Rb^{mk \times n \ell}$ and is defined as
\begin{equation}
	\Abf \otimes \Bbf \defeq
	\begin{bmatrix}
		\Abf(1,1) \Bbf & \Abf(1,2) \Bbf & \cdots & \Abf(1,n) \Bbf \\
		\Abf(2,1) \Bbf & \Abf(2,2) \Bbf & \cdots & \Abf(2,n) \Bbf \\
		\vdots & \vdots & & \vdots \\
		\Abf(m,1) \Bbf & \Abf(m,2) \Bbf & \cdots & \Abf(m,n) \Bbf \\
	\end{bmatrix}.
\end{equation}
The \emph{Khatri--Rao product}, sometimes called the columnwise Kronecker product, of two matrices $\Abf \in \Rb^{m \times n}$ and $\Bbf \in \Rb^{k \times n}$ is denoted by $\Abf \odot \Bbf \in \Rb^{mk \times n}$ and is defined as
\begin{equation}
	\Abf \odot \Bbf \defeq
	\begin{bmatrix}
		\Abf(:,1) \otimes \Bbf(:,1) & \Abf(:,2) \otimes \Bbf(:,2) & \cdots & \Abf(:,n) \otimes \Bbf(:,n)
	\end{bmatrix}.
\end{equation}

It is not obvious how to extend the standard asymptotic notation for single-variable functions to multi-variable functions \citep{howell2008AsymptoticNotation}.
Suppose $f$ and $g$ are positive functions over some parameters $x_1, \ldots, x_n$.
We say that a function $f(x_1,\ldots,x_n)$ is $O(g(x_1, \ldots, x_n))$ if there exists a constant $C > 0$ such that $f(x_1, \ldots, x_n) \leq C g(x_1, \ldots, x_n)$ for all valid values of the parameters $x_1, \ldots, x_n$.
The notation $\tilde{O}$ means the same as $O$ but with polylogarithmic factors ignored.
We say that a function $f(x_1, \ldots, x_n)$ is $\Omega(g(x_1, \ldots, x_n))$, or alternatively write $f(x_1, \ldots, x_n) \gtrsim g(x_1, \ldots, x_n)$, if there exists a constant $C > 0$ such that $f(x_1, \ldots, x_n) \geq C g(x_1, \ldots, x_n)$ for all valid values of the parameters $x_1, \ldots, x_n$.

\section{Missing Proofs} \label{sec:proofs}

\subsection{Proof of Theorem~\ref{thm:leverage-score-estimation}} \label{sec:proof-leverage-score-estimation}

We first state and prove Lemma~\ref{lemma:subspace-embedding-properties}.
It is similar to Lemma~5 in \citet{drineas2012FastApproximation} and Lemma~4.1 in \citet{drineas2006SamplingAlgorithms} which consider the case when $\Psibf$ is a fast Johnson--Lindenstrauss transform and a sampling matrix, respectively, instead of a subspace embedding.
\begin{lemma} \label{lemma:subspace-embedding-properties}
	Consider a matrix $\Abf \in \Rb^{I \times R}$ where $I > R$. 
	Let $\Abf = \Ubf \Sigmabf \Vbf^{\top}$ be a compact SVD of $\Abf$.
	Suppose $\Psibf$ is a $\gamma$-subspace embedding for $\Abf$ with $\gamma \in (0,1)$, and let $\Psibf \Ubf = \Qbf \Lambdabf \Wbf^\top$ be a compact SVD.
	Then, the following hold:
	\begin{enumerate}[(i)]
		\item $\rank(\Psibf \Abf) = \rank(\Psibf \Ubf) = \rank(\Abf)$,
		\item $\| \Ibf - \Lambdabf^{-2} \|_2 \leq \gamma/(1-\gamma)$,
		\item $(\Psibf \Abf)^\dagger = \Vbf \Sigmabf^{-1} (\Psibf \Ubf)^\dagger$.
	\end{enumerate}
\end{lemma}

\begin{proof}
	The proof follows similar arguments as those used in the proof of Lemma~4.1 in \citet{drineas2006SamplingAlgorithms}.
	Since $\Psibf$ is a $\gamma$-subspace embedding for $\Abf$, we have that
	\begin{equation}
		(1-\gamma) \|\Ubf \Sigmabf \Vbf^\top \xbf\|_2^2 \leq \| \Psibf \Ubf \Sigmabf \Vbf^\top \xbf \|_2^2 \leq (1+\gamma) \| \Ubf \Sigmabf \Vbf^\top \xbf \|_2^2 \;\;\;\; \text{for all } \xbf \in \Rb^R.
	\end{equation}
	Let $r \defeq \rank(\Abf)$. 
	Since $\Sigmabf \Vbf^\top \in \Rb^{r \times R}$ is full rank, and using unitary invariance of the spectral norm, it follows that
	\begin{equation}
		(1-\gamma) \|\ybf\|_2^2 \leq \| \Psibf \Ubf \ybf \|_2^2 \leq (1+\gamma) \| \ybf \|_2^2 \;\;\;\; \text{for all } \ybf \in \Rb^r.
	\end{equation}
	Using Theorem~8.6.1 in \citet{golub2013MatrixComputations}, this in turn implies that
	\begin{equation} \label{eq:singular-value-bound-Pi-U}
		1-\gamma \leq \sigma_i^2(\Psibf \Ubf) \leq 1+\gamma \;\;\;\; \text{for all } i \in [r].
	\end{equation}
	Consequently, $\rank(\Psibf \Ubf) = r = \rank(\Abf)$.
	Moreover, since $\rank(\Psibf \Abf) = \rank(\Psibf \Ubf \Sigmabf \Vbf^\top)$ and $\Sigmabf \Vbf^\top$ is full rank, it follows that $\rank(\Psibf \Abf) = \rank(\Psibf \Ubf)$.
	This completes the proof of (i).
	
	Next, note that
	\begin{equation}
		\| \Ibf - \Lambdabf^{-2} \|_2 
		= \max_{i \in [r]} \Big| 1 - \frac{1}{\sigma_i^2(\Psibf \Ubf)} \Big| 
		= \max_{i \in [r]} \Big| \frac{\sigma_i^2(\Psibf \Ubf) - 1}{\sigma_i^2(\Psibf \Ubf)} \Big| \leq \frac{\gamma}{1-\gamma},
	\end{equation}
	where the inequality follows from the bound in \eqref{eq:singular-value-bound-Pi-U}.
	This completes the proof of (ii).
	
	We may write
	\begin{equation} \label{eq:SA-dagger}
		(\Psibf \Abf)^\dagger 
		= (\Qbf \Lambdabf \Wbf^{\top} \Sigmabf \Vbf^\top)^\dagger 
		= \Vbf (\Lambdabf \Wbf^{\top} \Sigmabf)^\dagger \Qbf^\top  
	\end{equation}
	where the second equality follows since $\Qbf$ and $\Vbf$ have orthonormal columns.
	Since $\rank(\Psibf \Ubf) = \rank(\Abf)$ due to (i), the matrix $\Lambdabf \Wbf^\top \in \Rb^{r \times r}$ is invertible, and therefore $\Lambdabf \Wbf^\top \Sigmabf \in \Rb^{r \times r}$ is invertible, and hence
	\begin{equation} \label{eq:invertible-matrix}
		(\Lambdabf \Wbf^\top \Sigmabf)^\dagger 
		= \Sigmabf^{-1} \Wbf \Lambdabf^{-1}.
	\end{equation}
	Consequently, combining \eqref{eq:SA-dagger} and \eqref{eq:invertible-matrix} we have
	\begin{equation}
		(\Psibf \Abf)^\dagger 
		= \Vbf \Sigmabf^{-1} \Wbf \Lambdabf^{-1} \Qbf^\top 
		= \Vbf \Sigmabf^{-1} (\Psibf \Ubf)^\dagger.
	\end{equation}
	This completes the proof of (iii).
\end{proof}

We are now ready to prove the statement in Theorem~\ref{thm:leverage-score-estimation}.
\begin{proof}[Proof of Theorem~\ref{thm:leverage-score-estimation}]
	Our proof is similar to the proof of Lemma~9 in \citet{drineas2012FastApproximation}.
	Let $\Abf = \Ubf \Sigmabf \Vbf^\top$ be a compact SVD, $r \defeq \rank(\Abf)$ and suppose $i \in [I]$.
	From Definition~\ref{def:leverage-score}, we have 
	\begin{equation} \label{eq:ell-i}
		\ell_i(\Abf) = \|\Ubf(i,:)\|_2^2 = \ebf_i^\top \Ubf \Ubf^\top \ebf_i.	
	\end{equation}
	Moreover, 
	\begin{equation} \label{eq:hat-ell-i}
		\tilde{\ell}_i(\Abf) 
		= \| \ebf_i^\top \Abf \Vbf_1 \Sigmabf_1^{-1} \Ubf_1^\top\|_2^2 
		= \| \ebf_i^\top \Abf (\Psibf \Abf)^\dagger \|_2^2 
		= \| \ebf_i^\top \Ubf (\Psibf \Ubf)^\dagger \|_2^2 
		= \ebf_i^\top \Ubf (\Psibf \Ubf)^\dagger (\Psibf \Ubf)^{\dagger\top} \Ubf^\top \ebf_i,	
	\end{equation}
	where the first equality follows from the definition of $\tilde{\ell}_i(\Abf)$ in \eqref{eq:u-tilde} and the unitary invariance of the spectral norm, and the third equality follows from Lemma~\ref{lemma:subspace-embedding-properties} (iii).
	From \eqref{eq:ell-i} and \eqref{eq:hat-ell-i}, we have
	\begin{equation} 
		\begin{aligned}	\label{eq:1-pf-lemma-dist-U-u-hat}
			| \ell_i(\Abf) - \tilde{\ell}_i(\Abf) | 
			&= | \ebf_i^\top \Ubf \big( \Ibf - (\Psibf \Ubf)^\dagger (\Psibf \Ubf)^{\dagger \top} \big) \Ubf^\top \ebf_i | \\
			&\leq \| \ebf_i^\top \Ubf \|_2 \cdot \| ( \Ibf - (\Psibf \Ubf)^\dagger (\Psibf \Ubf)^{\dagger \top} ) \Ubf^\top \ebf_i \|_2 \\
			&\leq \| \ebf_i^\top \Ubf \|_2 \cdot \| \Ibf - (\Psibf \Ubf)^\dagger (\Psibf \Ubf)^{\dagger \top} \|_2 \cdot \| \Ubf^\top \ebf_i \|_2 \\
			&= \| \Ibf - (\Psibf \Ubf)^\dagger (\Psibf \Ubf)^{\dagger \top} \|_2 \cdot \ell_i(\Abf),
		\end{aligned}
	\end{equation}
	where the first inequality follows from Cauchy--Schwarz inequality, and the second inequality follows from the definition of the matrix spectral norm.
	Let $\Psibf \Ubf = \Qbf \Lambdabf \Wbf^\top$ be a compact SVD.
	It follows that 
	\begin{equation} \label{eq:2-pf-lemma-dist-U-u-hat}
		\| \Ibf - (\Psibf \Ubf)^\dagger (\Psibf \Ubf)^{\dagger \top} \|_2 = \| \Ibf - \Wbf \Lambdabf^{-2} \Wbf^\top \|_2.
	\end{equation}
	From Lemma~\ref{lemma:subspace-embedding-properties} (i), it follows that $\Wbf$ is $r \times r$, hence $\Wbf \Wbf^\top = \Ibf$. 
	Consequently, and using unitary invariance of the spectral norm,
	\begin{equation} \label{eq:3-pf-lemma-dist-U-u-hat}
		\| \Ibf - \Wbf \Lambdabf^{-2} \Wbf^\top \|_2 = \|\Ibf - \Lambdabf^{-2}\|_2.
	\end{equation}
	Combining \eqref{eq:1-pf-lemma-dist-U-u-hat}, \eqref{eq:2-pf-lemma-dist-U-u-hat} and \eqref{eq:3-pf-lemma-dist-U-u-hat}, we get
	\begin{equation}
		\big| \ell_i(\Abf) - \tilde{\ell}_i(\Abf) \big| 
		\leq \|\Ibf - \Lambdabf^{-2}\|_2 \cdot \ell_i(\Abf) \leq \frac{\gamma}{1-\gamma} \ell_i(\Abf),
	\end{equation}
	where the last inequality follows from Lemma~\ref{lemma:subspace-embedding-properties} (ii).
	This completes the proof.
\end{proof}

\subsection{Proof of Theorem~\ref{thm:cp-main}} \label{sec:proof-cp-main}

We first state some results that we will need for this proof.
Lemma~\ref{lemma:kronecker-property-1} follows from Lemma~4.2.10 in \citet{horn1994TopicsMatrix}.
\begin{lemma} \label{lemma:kronecker-property-1}
	For matrices $\Mbf_1, \ldots, \Mbf_n$ and $\Nbf_1, \ldots, \Nbf_n$ of appropriate sizes,
	\begin{equation}
		(\Mbf_1 \otimes \cdots \otimes \Mbf_n) \cdot (\Nbf_1 \otimes \cdots \otimes \Nbf_n) = (\Mbf_1 \Nbf_1) \otimes \cdots \otimes (\Mbf_n \Nbf_n).
	\end{equation}
\end{lemma}
Theorem~\ref{thm:recursive-sketch} follows directly from Theorem~1 in \citet{ahle2020ObliviousSketching} and its proof.%
\footnote{
	Since we are not considering regularized least squares problems, the statistical dimension $s_\lambda$ in \citet{ahle2020ObliviousSketching} just becomes equal to the number of columns of $\Abf$, which is $R$ in our case.
	The statement of Theorem~1 in \citet{ahle2020ObliviousSketching} uses $\delta=1/10$, but the statement for general $\delta$ is easy to infer from their proof of the theorem.
}
\begin{theorem} \label{thm:recursive-sketch}
	Let $\Abf \in \Rb^{I^N \times R}$. 
	Let $\Psibf \in \Rb^{J \times I^N}$ be the $(J, (I)_{j=1}^N)$-recursive sketch described in Section~\ref{sec:gaussian-and-recursive-sketches}.
	If $J \gtrsim N R^2 /(\gamma^2 \delta)$, then $\Psibf$ is a $\gamma$-subspace embedding for $\Abf$ with probability at least $1-\delta$.
\end{theorem}

It is easy to generalize Theorem~\ref{thm:recursive-sketch} to the setting when $\Psibf$ is a $(J, (I_j)_{j=1}^N)$-recursive sketch where the $I_j$ are not necessarily all equal.
\begin{corr} \label{corr:recursive-sketch}
	Let $\Abf \in \Rb^{\Pi_{j=1}^N I_j \times R}$.
	Let $\Psibf \in \Rb^{J \times \Pi_{j=1}^N I_j}$ be the $(J, (I_j)_{j=1}^N)$-recursive sketch described in Section~\ref{sec:gaussian-and-recursive-sketches}.
	If $J \gtrsim N R^2 / (\gamma^2 \delta)$, then $\Psibf$ is a $\gamma$-subspace embedding for $\Abf$ with probability at least $1-\delta$.
\end{corr}
\begin{proof}
	Let $q \defeq \lceil \log_2(N) \rceil$, $I_{\max} \defeq \max_{j \in [N]} I_j$, and $\tilde{I}_j \defeq I_j$ for $j \leq N$ and $\tilde{I}_j \defeq I_{\max}$ for $j > N$.
	Let $\onebf_{1 \times R}$ denote a length-$R$ row vector of all ones. 
	From the definition of the recursive sketch in Section~\ref{sec:gaussian-and-recursive-sketches} and the factorization in \eqref{eq:recursive-sketch-factorization}, we have
	\begin{equation} \label{eq:Psibf-A}
	\begin{aligned}
		\Psibf \Abf 
		&= \Psibf_{J, (\tilde{I}_j)_{j=1}^{2^q}} \big( \Abf \odot (\ebf_1^{\otimes(2^q - N)} \onebf_{1 \times R}) \big) \\
		&= \Tbf^{(q)} \Tbf^{(q-1)} \cdots \Tbf^{(1)} \Big( \bigotimes_{j=1}^{2^q} \Cbf_j \Big) \big( \Abf \odot (\ebf_1^{\otimes(2^q - N)} \onebf_{1 \times R}) \big) \\
		&= \Tbf^{(q)} \Tbf^{(q-1)} \cdots \Tbf^{(1)} \bigg( \Big( \Big(\bigotimes_{j=1}^N \Cbf_{j}\Big) \Abf \Big) \odot \Big( \Big( \bigotimes_{j=N+1}^{2^q} \Cbf_{j} \ebf_1 \Big) \onebf_{1 \times R} \Big) \bigg),
	\end{aligned}
	\end{equation}
	where the last equality follows from Lemma~\ref{lemma:kronecker-property-1}.
	Define two index sets 
	\begin{equation} 
		\Ic \defeq [I_1] \times \cdots \times [I_N] \;\;\;\; \text{and} \;\;\;\; \Ic^c \defeq [I_{\max}]^N \setminus \Ic.
	\end{equation}
	Let $\hat{\Abf} \in \Rb^{I_{\max}^N \times R}$ be an augmented version of $\Abf$ defined as
	\begin{equation} \label{eq:augmented-A}
		\hat{\Abf}(\overline{i_N \cdots i_1}, :) = 
		\begin{cases}
			\Abf(\overline{i_N \cdots i_1}, :) & \text{if } (i_1, \ldots, i_N) \in \Ic, \\
			0 & \text{if } (i_1, \ldots, i_N) \in \Ic^c.
		\end{cases}
	\end{equation}
	Let $\hat{\Psibf} \in \Rb^{J \times I_{\max}^N}$ be the $(J, (I_{\max})_{j=1}^{N})$-recursive sketch which uses independent CountSketches defined as
	\begin{equation} \label{eq:C-splitting}
		\hat{\Cbf}_j \in \Rb^{J \times I_{\max}} 
		\defeq
		\begin{bmatrix}
			\Cbf_j & \tilde{\Cbf}_j
		\end{bmatrix}, \;\;\;\; j \in [2^q],
	\end{equation}
	where the matrices $\Cbf_j$ are the same as in \eqref{eq:Psibf-A}, and each $\tilde{\Cbf}_j$ is an independent CountSketch of size $J \times (I_{\max} - I_j)$. 
	Again using the factorization in \eqref{eq:recursive-sketch-factorization}, we have
	\begin{equation} \label{eq:hat-Psibf-hat-A}
	\begin{aligned}
		\hat{\Psibf} \hat{\Abf} 
		&= \Psibf_{J, (I_{\max})_{j=1}^{2^q}} (\hat{\Abf} \odot \big(\ebf_1^{\otimes (2^q - N)} \onebf_{1 \times R})\big) \\
		&= \Tbf^{(q)} \Tbf^{(q-1)} \cdots \Tbf^{(1)} \Big(\bigotimes_{j=1}^{2^q} \hat{\Cbf}_j \Big) \big(\hat{\Abf} \odot (\ebf_1^{\otimes (2^q - N)} \onebf_{1 \times R})\big) \\
		&= \Tbf^{(q)} \Tbf^{(q-1)} \cdots \Tbf^{(1)} \bigg( \Big(\Big(\bigotimes_{j=1}^N \hat{\Cbf}_{j}\Big) \hat{\Abf}\Big) \odot \Big( \Big(\bigotimes_{j=N+1}^{2^q} \hat{\Cbf}_{j} \ebf_1 \Big) \onebf_{1 \times R} \Big) \bigg) ,
	\end{aligned}
	\end{equation}
	where the last equality follows from Lemma~\ref{lemma:kronecker-property-1}.
	From the definition of matrix multiplication, we have
	\begin{equation} \label{eq:Kronecker-C-applied-to-aug-A}
		\Big(\bigotimes_{j=1}^N \hat{\Cbf}_{j}\Big) \hat{\Abf} 
		= \sum_{(i_1, \ldots, i_N) \in \Ic \cup \Ic^c} \Big(\bigotimes_{j=1}^N \hat{\Cbf}_{j}\Big)(:, \overline{i_N \cdots i_1}) \hat{\Abf}(\overline{i_N \cdots i_1}, :).
	\end{equation}
	Due to \eqref{eq:C-splitting}, it follows that $\hat{\Cbf}_j(:, i_j) = \Cbf_j(:, i_j)$ when $i_j \in [I_j]$, and consequently 
	\begin{equation} \label{eq:C-splitting-effect}
		\Big(\bigotimes_{j=1}^N \hat{\Cbf}_{j}\Big)(:, \overline{i_N \cdots i_1}) 
		= \bigotimes_{j=1}^{N}\hat{\Cbf}_j(:, i_j) 
		= \bigotimes_{j=1}^{N}\Cbf_j(:, i_j)
		= \Big(\bigotimes_{j=1}^N \Cbf_j \Big)(:, \overline{i_N \cdots i_1}) \;\;\;\; \text{for all } (i_1,\ldots,i_N) \in \Ic.
	\end{equation}
	Using \eqref{eq:augmented-A} and \eqref{eq:C-splitting-effect}, we can simplify \eqref{eq:Kronecker-C-applied-to-aug-A} to 
	\begin{equation} \label{eq:equality-in-dist-1}
		\Big(\bigotimes_{j=1}^N \hat{\Cbf}_{j}\Big) \hat{\Abf} 
		= \sum_{(i_1, \ldots, i_N) \in \Ic} \Big(\bigotimes_{j=1}^N \Cbf_j\Big)(:, \overline{i_N \cdots i_1}) \Abf(\overline{i_N \cdots i_1}, :)
		= \Big(\bigotimes_{j=1}^N \Cbf_{j} \Big) \Abf.
	\end{equation}
	Similarly, since the first column of each $\hat{\Cbf}_j$ and $\Cbf_j$ are the same,
	\begin{equation} \label{eq:equality-in-dist-2}
		\Big( \bigotimes_{j=N+1}^{2^q} \hat{\Cbf}_{j} \ebf_1 \Big) \onebf_{1 \times R} = \Big( \bigotimes_{j=N+1}^{2^q} \Cbf_{j} \ebf_1 \Big) \onebf_{1 \times R}.
	\end{equation}
	Equations \eqref{eq:Psibf-A}, \eqref{eq:hat-Psibf-hat-A}, \eqref{eq:equality-in-dist-1} and \eqref{eq:equality-in-dist-2} together now imply that
	\begin{equation} \label{eq:equality-in-dist-3}
		\Psibf \Abf = \hat{\Psibf} \hat{\Abf}.
	\end{equation}
	Moreover, it follows immediately from \eqref{eq:augmented-A} that
	\begin{equation} \label{eq:equality-in-norm}
		\| \Abf \xbf \|_2 = \| \hat{\Abf} \xbf \|_2 \;\;\;\; \text{for all } \xbf \in \Rb^R.
	\end{equation}
	Theorem~\ref{thm:recursive-sketch} implies that 
	\begin{equation} 
		\Pb\big( \big| \|\hat{\Psibf} \hat{\Abf} \xbf\|_2^2 - \| \hat{\Abf} \xbf \|_2^2 \big| \leq \gamma \|\hat{\Abf} \xbf\|_2^2 \;\; \text{for all } \xbf \in \Rb^R \big) \geq 1-\delta.
	\end{equation}
	Due to \eqref{eq:equality-in-dist-3} and \eqref{eq:equality-in-norm}, this implies that
	\begin{equation}
		\Pb\big( \big| \|\Psibf \Abf \xbf\|_2^2 - \| \Abf \xbf \|_2^2 \big| \leq \gamma \|\Abf \xbf\|_2^2 \;\; \text{for all } \xbf \in \Rb^R \big) \geq 1-\delta,
	\end{equation}
	which is what we wanted to show.
\end{proof}

Theorem~\ref{thm:least-squares-guarantees} is a well-known result.
Since slightly different variations of it have appeared in the literature \citep{drineas2006SamplingAlgorithms, drineas2008RelativeerrorCUR, drineas2011FasterLeast, larsen2020PracticalLeverageBased} we provide a proof sketch just to give the reader some idea of how to derive the version we use.
\begin{theorem} \label{thm:least-squares-guarantees}
	Let $\Abf \in \Rb^{I \times R}$ be a matrix, and suppose $\Sbf \sim \Dc(J, \qbf)$ is a leverage score sampling matrix for $(\Abf, \beta)$ where $\beta \in (0,1]$, and that $\varepsilon, \delta \in (0,1)$. 
	Moreover, define $\OPT \defeq \min_{\Xbf} \|\Abf \Xbf - \Ybf\|_\F$ and $\tilde{\Xbf} \defeq \argmin_{\Xbf} \| \Sbf \Abf \Xbf - \Sbf \Ybf\|_\F$.
	If 
	\begin{equation} \label{eq:least-squares-guarantees-sketch-rate}
		J > \frac{4R}{\beta} \max\Big( \frac{4}{3(\sqrt{2}-1)^2}\ln\Big( \frac{4R}{\delta} \Big), \frac{1}{\varepsilon \delta} \Big),
	\end{equation}
	then the following holds with probability at least $1-\delta$:
	\begin{equation} \label{eq:relative-error-guarantee}
		\| \Abf \tilde{\Xbf} - \Ybf \|_\F \leq (1+\varepsilon) \OPT.
	\end{equation}
\end{theorem}
\begin{proof}[Proof sketch]
	Let $\Ubf \in \Rb^{I \times \rank(\Abf)}$ contain the left singular vectors of $\Abf$, and define $\Ybf^\perp \defeq (\Ibf - \Ubf\Ubf^\top) \Ybf$.
	According to a matrix version\footnote{See Lemma~S1 in \citet{malik2021SamplingBasedMethod}.} of Lemma~1 by \citet{drineas2011FasterLeast}, the statement in \eqref{eq:relative-error-guarantee} holds if both
	\begin{equation} \label{eq:cond-1-sketch}
		\sigma_{\min}^2(\Sbf \Ubf) \geq \frac{1}{\sqrt{2}}
	\end{equation}
	and
	\begin{equation} \label{eq:cond-2-sketch}
		\| \Ubf^\top \Sbf^\top \Sbf \Ybf^\perp \|_\F^2 \leq \frac{\varepsilon}{2} \OPT^2.
	\end{equation}
	To complete the proof, it is therefore sufficient to show that $\Sbf$ satisfies both \eqref{eq:cond-1-sketch} and \eqref{eq:cond-2-sketch} with probability at least $1-\delta$.
	Using Lemma~S2 in \citet{malik2021SamplingBasedMethod}, which is the same as Theorem~2.11 in \citet{woodruff2014SketchingTool} but with a slightly smaller constant, one can show that the condition \eqref{eq:cond-1-sketch} is satisfied with probability at least $1-\delta/2$ if 
	\begin{equation} \label{eq:J-partial-1}
		J > \frac{16}{3 (\sqrt{2}-1)^2} \frac{R}{\beta} \ln\Big(\frac{4R}{\delta}\Big).
	\end{equation}
	Next, using Lemma~8 in \citet{drineas2006FastMonte}, it follows that
	\begin{equation}
		\Eb\| \Ubf^\top \Sbf^\top \Sbf \Ybf^\perp \|_\F^2 \leq \frac{1}{J \beta} R \cdot \OPT^2.
	\end{equation}
	Markov's inequality together with the assumption
	\begin{equation} \label{eq:J-partial-2}
		J > \frac{4R}{\beta \varepsilon \delta}
	\end{equation}
	then implies that
	\begin{equation}
		\Pb\Big(\|\Ubf^\top \Sbf^\top \Sbf \Ybf^\perp\|_\F^2 > \frac{\varepsilon}{2} \OPT^2\Big) \leq \frac{2R}{J\varepsilon\beta} < \frac{\delta}{2}.
	\end{equation}
	If \eqref{eq:least-squares-guarantees-sketch-rate} is satisfied, then both \eqref{eq:J-partial-1} and \eqref{eq:J-partial-2} are satisfied, and consequently \eqref{eq:cond-1-sketch} and \eqref{eq:cond-2-sketch} are both true with probability at least $1-\delta$.
\end{proof}

We are now ready to prove the statement in Theorem~\ref{thm:cp-main}.

\begin{proof}[Proof of Theorem~\ref{thm:cp-main}]
Let $E_1$ denote the event that $\Psibf$ is a $1/3$-subspace embedding for $\Abf^{\neq n}$.
Following the notation used in Theorem~\ref{thm:leverage-score-estimation}, let $\Psibf \Abf^{\neq n} = \Ubf_1 \Sigmabf_1 \Vbf_1^{\top}$ be a compact SVD.
Let $E_2$ denote the event that \eqref{eq:cp-main-relative-error} is true.

According to Corollary~\ref{corr:recursive-sketch}, we can guarantee that $\Pb(E_1) \geq 1-\delta/2$ if we choose $J_1$ as in \eqref{eq:cp-J1}.
With $\gamma=1/3$ in Theorem~\ref{thm:leverage-score-estimation}, the estimates $\tilde{\ell}_i(\Abf^{\neq n})$ satisfy 
\begin{equation} \label{eq:bound-on-elli}
	\frac{1}{2} \ell_i(\Abf^{\neq n}) \leq \tilde{\ell}_i(\Abf^{\neq n}) \leq \frac{3}{2} \ell_i(\Abf^{\neq n}).
\end{equation}
Consequently, 
\begin{equation} \label{eq:elli-sum}
	\sum_{i=1}^{\Pi_{j \neq n} I_j} \tilde{\ell}_i(\Abf^{\neq n}) \leq \frac{3}{2} \sum_{i=1}^{\Pi_{j \neq n} I_j} \ell_i(\Abf^{\neq n}) = \frac{3}{2} \rank(\Abf^{\neq n}).
\end{equation}
Therefore, since $\qbf(i) \propto \tilde{\ell}_i(\Abf^{\neq n})$, it follows by combining \eqref{eq:bound-on-elli} and \eqref{eq:elli-sum} that
\begin{equation}
	\qbf(i) = \frac{\tilde{\ell}_i(\Abf^{\neq n})}{\sum_{i=1}^{\Pi_{j \neq n} I_j} \tilde{\ell}_i(\Abf^{\neq n})} \geq \frac{1}{3} \frac{\ell_i(\Abf^{\neq n})}{\rank(\Abf^{\neq n})}.
\end{equation}
In view of Definition~\ref{def:leverage-score-sampling}, Theorem~\ref{thm:leverage-score-estimation} therefore implies that $\Sbf \sim \Dc(J_2, \qbf)$ is a leverage score sampling matrix for $(\Abf^{\neq n}, 1/3)$ if the event $E_1$ is true.
From Theorem~\ref{thm:least-squares-guarantees}, it then follows that $\Pb(E_2 \mid E_1 ) \geq 1 - \delta/2$ if $J_2$ is chosen as in \eqref{eq:cp-J3}.
With the choices of $J_1$ and $J_2$ above we now have
\begin{equation} \label{eq:event-formula}
	\Pb(E_2) \geq \Pb(E_1, E_2) = \Pb(E_1) \Pb(E_2 \mid E_1) \geq (1-\delta/2)^2 \geq 1-\delta
\end{equation}
which is what we wanted to show.
\end{proof}

\subsection{Proof of Lemma~\ref{lem:cp-sampling-results}} \label{sec:proof-cp-sampling-results}

Recall that $\Phibf \defeq \Vbf_1 \Sigmabf_1^{-1} (\Vbf_1 \Sigmabf_1^{-1})^\top$, where $\Psibf \Abf^{\neq n} = \Ubf_1 \Sigmabf_1 \Vbf_1^\top$ is a compact SVD.
From \eqref{eq:u-tilde} we have
\begin{equation} \label{eq:cp-lev-score-estimate}
	\tilde{\ell_i}(\Abf^{\neq n}) 
	= \ebf_i^\top \Abf^{\neq n} \Phibf \Abf^{\neq n \top} \ebf_i 
	= (\Abf^{\neq n} \Phibf \Abf^{\neq n \top})(i,i) 
	= \sum_{r, k} \Phibf(r,k) \cdot \prod_{j \neq n} \Abf^{(j)}(i_j,r) \Abf^{(j)}(i_j,k),
\end{equation}
where the last equality follows from the definition of $\Abf^{\neq n}$ in \eqref{eq:cp-design-matrix}, and $i = \overline{i_1 \cdots i_{n-1} i_{n+1} \cdots i_N}$.
Using \eqref{eq:cp-lev-score-estimate}, we can compute the normalization constant $C$ as
\begin{equation}
	C 
	\defeq \sum_{i} \tilde{\ell}_i(\Abf^{\neq n}) 
	= \sum_{r, k} \Phibf(r,k) \cdot \prod_{j \neq n} \sum_{i_j} \Abf^{(j)}(i_j,r) \Abf^{(j)}(i_j,k)
	= \sum_{r, k} \Phibf(r,k) \cdot \prod_{j \neq n} (\Abf^{(j)\top} \Abf^{(j)})(r,k),
\end{equation}
which proves \eqref{eq:cp-sampling-results-C}.

To make notation a bit less cumbersome, we will use the abbreviated notation
\begin{equation}
	\sum_{\{i_j\}_{j>m, j \neq n}} \;\;\;\; \text{to denote} \;\;\;\;
	\begin{cases}
		\sum_{i_{m+1}} \cdots \sum_{i_{n-1}} \sum_{i_{n+1}} \cdots \sum_{i_N} & \text{if } n > m, \\
		\sum_{i_{m+1}} \cdots \sum_{i_N} & \text{otherwise}.
	\end{cases}
\end{equation}
Similar abbreviated notation will also be used later on for other indices.
We can again use \eqref{eq:cp-lev-score-estimate} to compute the marginal probabilities of drawing $(i_j)_{j \leq m, j \neq n}$ as
\begin{equation} \label{eq:cp-sample-subsequent-indices}
\begin{aligned}
	\Pb((i_j)_{j \leq m, j \neq n}) 
	&= \frac{1}{C} \sum_{\{ i_j \}_{j > m, j \neq n}} \tilde{\ell}_i(\Abf^{\neq n}) \\
	&= \frac{1}{C} \sum_{\{ i_j \}_{j > m, j \neq n}} \Big( \sum_{r, k} \Phibf(r,k) \prod_{j \neq n} \Abf^{(j)}(i_j, r) \Abf^{(j)}(i_j,k) \Big) \\
	&= \frac{1}{C} \sum_{r,k} \Phibf(r,k)  \Big(\prod_{\substack{j \leq m \\ j \neq n}} \Abf^{(j)}(i_j, r) \Abf^{(j)}(i_j,k) \Big)  \Big(\prod_{\substack{j > m \\ j \neq n}} \big(\Abf^{(j)\top} \Abf^{(j)} \big)(r,k) \Big),
\end{aligned}
\end{equation}
which proves \eqref{eq:cp-sampling-results-joint-prob}.

\subsection{Proof of Theorem~\ref{thm:tr-main}} \label{sec:proof-tr-main}

The strategy of this proof is similar to that for the proof of Theorem~\ref{thm:cp-main} given in Section~\ref{sec:proof-cp-main}.
Let $E_1$ denote the event that $\Psibf$ is a $1/3$-subspace embedding for $\Gbf_{[2]}^{\neq n}$.
Following the notation used in Theorem~\ref{thm:leverage-score-estimation}, let $\Psibf \Gbf_{[2]}^{\neq n} = \Ubf_1 \Sigmabf_1 \Vbf_1^\top$ be a compact SVD.
Let $E_2$ denote the event that \eqref{eq:tr-main-relative-error} is true.

The matrix $\Gbf_{[2]}^{\neq n}$ is of size $\prod_{j \neq n} I_j \times R_{n-1} R_n$. 
According to Corollary~\ref{corr:recursive-sketch}, we can therefore guarantee that $\Pb(E_1) \geq 1 - \delta/2$ if we choose $J_1$ as in \eqref{eq:tr-J1}.
Following the same line of reasoning as in the proof of Theorem~\ref{thm:cp-main}, we can show that the choice $\gamma=1/3$ in Theorem~\ref{thm:leverage-score-estimation} combined with the fact $\qbf(i) \propto \tilde{\ell}_i(\Gbf_{[2]}^{\neq n})$ implies that
\begin{equation}
	\qbf(i) = \frac{\tilde{\ell}_i(\Gbf_{[2]}^{\neq n})}{\sum_{i=1}^{\Pi_{j \neq n} I_j} \tilde{\ell}_i(\Gbf_{[2]}^{\neq n})} \geq \frac{1}{3} \frac{\ell_i(\Gbf_{[2]}^{\neq n})}{\rank(\Gbf_{[2]}^{\neq n})}.
\end{equation}
In view of Definition~\ref{def:leverage-score-sampling}, Theorem~\ref{thm:leverage-score-estimation} therefore implies that $\Sbf \sim \Dc(J_2, \qbf)$ is a leverage score sampling matrix for $(\Gbf_{[2]}^{\neq n}, 1/3)$ if the event $E_1$ is true.
From Theorem~\ref{thm:least-squares-guarantees}, it then follows that $\Pb(E_2 \mid E_1) \geq 1 - \delta/2$ if $J_2$ is chosen as in \eqref{eq:tr-J3}.
With the choices of $J_1$ and $J_2$ above and the formula \eqref{eq:event-formula}, we have that $\Pb(E_2) \geq 1-\delta$, which is what we wanted to show.

\subsection{Proof of Lemma~\ref{lem:correctness-of-tr-sketch}} \label{sec:proof-correctness-of-tr-sketch}

It follows directly from Definitions~\ref{def:unfolding} and \ref{def:subchain} that
\begin{equation} \label{eq:unfolded-subchain-tensor-w-ordering}
	\Gbf_{[2]}^{\neq n}(\overline{i_{n+1} \cdots i_N i_1 \cdots i_{n-1}}, \overline{r_{n-1} r_n}) = \sum_{\{r_j\}_{j \neq n-1, n}} \prod_{j=1}^{N-1} \Gbf_{[2]}^{(\wbf(j))}(i_{\wbf(j)}, \overline{r_{\wbf(j)} r_{\wbf(j)-1}}),
\end{equation}
and therefore the columns of $\Gbf_{[2]}^{\neq n}$ can be written as
\begin{equation}
	\Gbf_{[2]}^{\neq n}(:, \overline{r_{n-1} r_n}) = \sum_{\{r_j\}_{j \neq n-1, n}} \bigotimes_{j=1}^{N-1} \Gbf_{[2]}^{(\wbf(j))}(:, \overline{r_{\wbf(j)} r_{\wbf(j)-1}}).
\end{equation}

Let $q \defeq \lceil \log_2 (N-1) \rceil$.
Using the definition of the recursive sketch in Section~\ref{sec:gaussian-and-recursive-sketches} and the factorization in \eqref{eq:recursive-sketch-factorization}, we have
\begin{equation} \label{eq:justification-tr-sketch-1}
	\Psibf \Gbf_{[2]}^{\neq n}(:, \overline{r_{n-1} r_n}) 
	= \Tbf^{(q)} \Tbf^{(q-1)} \cdots \Tbf^{(1)} \Cbf \sum_{\{r_j\}_{j \neq n-1, n}} \bigg(\Big(\bigotimes_{j=1}^{N-1} \Gbf_{[2]}^{(\wbf(j))}(:, \overline{r_{\wbf(j)} r_{\wbf(j)-1}}) \Big) \otimes \ebf_1^{\otimes(2^q - (N-1))} \bigg).
\end{equation}
The notation in the equation above is quite cumbersome. 
In particular, the ordering of the matrices $\Gbf_{[2]}^{(j)}$ in the Kronecker product is somewhat awkward. 
To alleviate the issue somewhat, we define $\Hbf^{(j)}$ for $j \in [2^q]$ as we did in Section~\ref{sec:tr-efficient-subspace-embedding}:
\begin{itemize}
	\item Let $\Hbf^{(1)} \in \Rb^{I_{n-1} \times R_{n-2}}$ be a matrix with columns $\Hbf^{(1)}(:, k) \defeq \Gbf_{[2]}^{(n-1)}(:, \overline{r_{n-1} k})$ for $k \in [R_{n-2}]$.
	\item Let $\Hbf^{(j)} \defeq \Gbf_{[2]}^{(\wbf(j))} \in \Rb^{I_{\wbf(j)} \times R_{\wbf(j)} R_{\wbf(j)-1}}$ for $2 \leq j \leq N-2$.
	\item Let $\Hbf^{(N-1)} \in \Rb^{I_{n+1} \times R_{n+1}}$ be a matrix with columns $\Hbf^{(N-1)}(:, k) \defeq \Gbf_{[2]}^{(n+1)}(:, \overline{k r_n})$ for $k \in [R_{n+1}]$.
	\item Let $\Hbf^{(j)} \defeq \ebf_1 \in \Rb^{\max_{j \neq n} I_j}$ be a column vector for $N \leq j \leq 2^q$.
\end{itemize}
Moreover, we also define the numbers $K^{(0)}_j$ for $j \in [2^q+1]$ as in Section~\ref{sec:tr-efficient-subspace-embedding}:
\begin{equation}
	K_j^{(0)} \defeq
	\begin{cases}
		R_{\wbf(j)} & \text{if } 2 \leq j \leq N-1, \\
		
		1 & \text{otherwise}.
	\end{cases}
\end{equation}
With this new notation, we can write \eqref{eq:justification-tr-sketch-1} as 
\begin{equation} \label{eq:justification-tr-sketch-2}
	\Psibf \Gbf_{[2]}^{\neq n}(:, \overline{r_{n-1} r_n}) 
	= \Tbf^{(q)} \Tbf^{(q-1)} \cdots \Tbf^{(1)} \Cbf \sum_{\{k_j\}_{j=1}^{2^{q}+1}} \bigotimes_{j=1}^{2^q} \Hbf^{(j)}(:, \overline{k_j k_{j+1}}),
\end{equation}
where each summation index $k_j$ goes over values $k_j \in [K_j^{(0)}]$.
Using Lemma~\ref{lemma:kronecker-property-1}, Equation \eqref{eq:justification-tr-sketch-2} can be written as
\begin{equation}
\begin{aligned} \label{eq:justification-tr-sketch-3}
	\Psibf \Gbf_{[2]}^{\neq n}(:, \overline{r_{n-1} r_n}) 
	&= \Tbf^{(q)} \Tbf^{(q-1)} \cdots \Tbf^{(1)} \sum_{\{k_j\}_{j=1}^{2^{q}+1}} \bigotimes_{j=1}^{2^q} \Cbf_j \Hbf^{(j)}(:, \overline{k_j k_{j+1}}) \\
	&= \Tbf^{(q)} \Tbf^{(q-1)} \cdots \Tbf^{(1)} \sum_{\{k_j\}_{j=1}^{2^{q}+1}} \bigotimes_{j=1}^{2^q} \Ybf_{j}^{(0)}(:, \overline{k_j k_{j+1}}) , \\
\end{aligned}
\end{equation}
where $\Ybf_j^{(0)}$ was defined in \eqref{eq:Y0j}. 
Recalling that $\Tbf^{(1)} \defeq \bigotimes_{j=1}^{2^{q-1}} \Tbf^{(1)}_j$, we may further rewrite \eqref{eq:justification-tr-sketch-3} as
\begin{equation}
\begin{aligned}
	\Psibf \Gbf_{[2]}^{\neq n}(:, \overline{r_{n-1} r_n}) 
	&= \Tbf^{(q)} \Tbf^{(q-1)} \cdots \Tbf^{(2)} \Big( \bigotimes_{j=1}^{2^{q-1}} \Tbf^{(1)}_j \Big) \sum_{\{k_j\}_{j=1}^{2^{q}+1}} \bigotimes_{j=1}^{2^{q-1}} (\Ybf^{(0)}_{2j-1}(:, \overline{k_{2j-1} k_{2j}}) \otimes \Ybf^{(0)}_{2j}(:, \overline{k_{2j} k_{2j+1}})) \\
	&= \Tbf^{(q)} \Tbf^{(q-1)} \cdots \Tbf^{(2)} \sum_{\{k_{2j-1}\}_{j=1}^{2^{q-1}+1}} \bigotimes_{j=1}^{2^{q-1}} \sum_{k_{2j}} \Tbf^{(1)}_j (\Ybf^{(0)}_{2j-1}(:, \overline{k_{2j-1} k_{2j}}) \otimes \Ybf^{(0)}_{2j}(:, \overline{k_{2j} k_{2j+1}})) \\
	&= \Tbf^{(q)} \Tbf^{(q-1)} \cdots \Tbf^{(2)} \sum_{\{k_{2j-1}\}_{j=1}^{2^{q-1}+1}} \bigotimes_{j=1}^{2^{q-1}} \Ybf^{(1)}_j(:, \overline{k_{2j-1} k_{2j+1}}), \\
\end{aligned}
\end{equation}
where the second equality follows from Lemma~\ref{lemma:kronecker-property-1}, and each $\Ybf^{(1)}_j$ is defined as in \eqref{eq:Ymj}.
Defining $K_j^{(1)} \defeq K_{2j-1}^{(0)}$ for $j \in [2^{q-1}+1]$, we can further rewrite the equation above as
\begin{equation} \label{eq:justification-tr-sketch-4}
	\Psibf \Gbf_{[2]}^{\neq n}(:, \overline{r_{n-1} r_n}) = \Tbf^{(q)} \Tbf^{(q-1)} \cdots \Tbf^{(2)} \sum_{\{k_{j}\}_{j=1}^{2^{q-1}+1}} \bigotimes_{j=1}^{2^{q-1}} \Ybf^{(1)}_j(:, \overline{k_{j} k_{j+1}}),
\end{equation}
where each summation index $k_j$ now goes over the values $k_j \in [K_j^{(1)}]$.
In general, for $m \in [q]$, we have
\begin{equation} \label{eq:justification-tr-sketch-5}
\begin{aligned}
	&\Tbf^{(q)} \Tbf^{(q-1)} \cdots \Tbf^{(m)} \sum_{\{k_{j}\}_{j=1}^{2^{q-m+1}+1}} \bigotimes_{j=1}^{2^{q-m+1}} \Ybf^{(m-1)}_j(:, \overline{k_{j} k_{j+1}}) \\
	&= \Tbf^{(q)} \Tbf^{(q-1)} \cdots \Tbf^{(m+1)} \sum_{\{\ell_{j}\}_{j=1}^{2^{q-m}+1}} \bigotimes_{j=1}^{2^{q-m}} \Ybf^{(m)}_j(:, \overline{\ell_{j} \ell_{j+1}}),
\end{aligned}
\end{equation}
where the summation indices $k_j$ and $\ell_j$ take on values $k_j \in [K_j^{(m-1)}]$ and $\ell_j \in [K_j^{(m)}]$, respectively, where $K_j^{(m)} \defeq K_{2j-1}^{(m-1)}$ for $j \in [2^{q-m}+1]$, and where each $\Ybf^{(m)}_j$ is defined as in \eqref{eq:Ymj}.
Combining \eqref{eq:justification-tr-sketch-4} and \eqref{eq:justification-tr-sketch-5}, it follows by induction that
\begin{equation}
	\Psibf \Gbf_{[2]}^{\neq n} (:, \overline{r_{n-1} r_n}) = \sum_{k_1 \in [K^{(q)}_1]} \sum_{k_2 \in [K^{(q)}_2]} \Ybf_{1}^{(q)}(:, \overline{k_1 k_2}) = \Ybf_{1}^{(q)},
\end{equation}
where the last equality follows since $K_1^{(q)} = K_2^{(q)} = 1$.

\subsection{Proof of Lemma~\ref{lem:tr-sampling-results}} \label{sec:proof-tr-sampling-results}

Throughout the following computations the summation indices go over $i = \overline{i_{n+1} \cdots i_N i_1 \cdots i_{n-1}} \in [\prod_{j \neq n} I_j]$ with $i_j \in [I_j]$ and $r_j, k_j \in [R_j]$ for each $j \in [N]$.
Recall that $\Phibf \defeq \Vbf_1 \Sigmabf_1^{-1} (\Vbf_1 \Sigmabf_1^{-1})^\top$, where $\Psibf \Gbf_{[2]}^{\neq n} = \Ubf_1 \Sigmabf_1 \Vbf_1^\top$ is a compact SVD.
From \eqref{eq:u-tilde} we have
\begin{equation} \label{eq:tr-lev-score-estimate}
\begin{aligned}
	\tilde{\ell}_i(\Gbf_{[2]}^{\neq n}) 
	&= \ebf_i^\top \Gbf_{[2]}^{\neq n} \Phibf \Gbf_{[2]}^{\neq n \top} \ebf_i 
	= \big( \Gbf_{[2]}^{\neq n} \Phibf \Gbf_{[2]}^{\neq n \top} \big)(i,i) \\
	&= \sum_{\substack{r_{n-1}, r_n \\ k_{n-1}, k_n}} \Gbf_{[2]}^{\neq n}(i, \overline{r_{n-1} r_n}) \Phibf(\overline{r_{n-1} r_n}, \overline{k_{n-1} k_n}) \Gbf_{[2]}^{\neq n}(i, \overline{k_{n-1} k_n}).
\end{aligned}
\end{equation}
From Definitions~\ref{def:unfolding} and \ref{def:subchain} it follows that%
\footnote{
	The only difference between \eqref{eq:unfolded-subchain-tensor-w-ordering} and \eqref{eq:unfolded-subchain-tensor} is that the terms in the product are arranged in a different order.
}
\begin{equation} \label{eq:unfolded-subchain-tensor}
	\Gbf_{[2]}^{\neq n}(i, \overline{r_{n-1} r_{n}})
	= \Gbf_{[2]}^{\neq n}(\overline{i_{n+1} \cdots i_{N} i_1 \cdots i_{n-1}}, \overline{r_{n-1} r_{n}})
	= \sum_{\{r_j\}_{j \neq n-1, n}} \prod_{j \neq n} \Gbf_{[2]}^{(j)}(i_{j}, \overline{r_{j} r_{j-1}}).  
\end{equation}
Using \eqref{eq:tr-lev-score-estimate} and \eqref{eq:unfolded-subchain-tensor}, we have
\begin{equation} \label{eq:tr-constant-1}
\begin{aligned}
	C 
	&\defeq \sum_i \tilde{\ell}_i(\Gbf_{[2]}^{\neq n}) \\ 
	&= \sum_i \sum_{\substack{r_{n-1}, r_n \\ k_{n-1}, k_n}} \Big(\sum_{\{r_j\}_{j \neq n-1, n}} \prod_{j \neq n} \Gbf_{[2]}^{(j)}(i_{j}, \overline{r_{j} r_{j-1}})\Big) \Phibf(\overline{r_{n-1} r_{n}}, \overline{k_{n-1} k_n}) \Big(\sum_{\{k_j\}_{j \neq n-1, n}} \prod_{j \neq n} \Gbf_{[2]}^{(j)}(i_{j}, \overline{k_{j} k_{j-1}})\Big) \\
	&= \sum_{\substack{r_1, \ldots, r_N \\ k_1, \ldots, k_N}} \Phibf(\overline{r_{n-1} r_n}, \overline{k_{n-1} k_n})  \prod_{j \neq n} \Big( \sum_{i_j} \Gbf_{[2]}^{(j)}(i_{j}, \overline{r_{j} r_{j-1}}) \Gbf_{[2]}^{(j)}(i_{j}, \overline{k_{j} k_{j-1}}) \Big) \\
	&= \sum_{\substack{r_1, \ldots, r_N \\ k_1, \ldots, k_N}} \Phibf(\overline{r_{n-1} r_n}, \overline{k_{n-1} k_n})  \prod_{j \neq n} \big(\Gbf_{[2]}^{(j) \top} \Gbf_{[2]}^{(j)}\big)(\overline{r_{j} r_{j-1}}, \overline{k_{j} k_{j-1}}),
\end{aligned}
\end{equation}
which proves the expression in \eqref{eq:tr-sampling-results-C}.

Moreover, using \eqref{eq:tr-lev-score-estimate} and \eqref{eq:unfolded-subchain-tensor} we have that the marginal probability of drawing $(i_j)_{j \leq m, j \neq n}$ is
\begin{equation}
\begin{aligned}
	&\Pb((i_j)_{j \leq m, j \neq n}) 
	= \frac{1}{C} \sum_{\{i_j\}_{j > m, j \neq n}} \tilde{\ell}_i(\Gbf_{[2]}^{\neq n}) \\
	&= \frac{1}{C} \sum_{\{i_j\}_{j > m, j \neq n}} \sum_{\substack{r_{n-1}, r_n \\ k_{n-1}, k_n}} \Phibf(\overline{r_{n-1} r_n}, \overline{k_{n-1} k_n}) \Big( \sum_{\{r_j\}_{j \neq n-1, n}} \prod_{j \neq n} \Gbf_{[2]}^{(j)}(i_j, \overline{r_{j} r_{j-1}}) \Big) \Big( \sum_{\{k_j\}_{j \neq n-1, n}} \prod_{j \neq n} \Gbf_{[2]}^{(j)}(i_j, \overline{k_{j} k_{j-1}}) \Big) \\
	&= \frac{1}{C} \sum_{\substack{r_1, \ldots, r_N \\ k_1, \ldots, k_N}} \Phibf(\overline{r_{n-1} r_n}, \overline{k_{n-1} k_n}) \Big( \prod_{\substack{j \leq m \\ j \neq n}} \Gbf_{[2]}^{(j)}(i_j, \overline{r_{j}r_{j-1}}) \Gbf_{[2]}^{(j)}(i_j, \overline{k_{j}k_{j-1}}) \Big) \Big( \prod_{\substack{j > m \\ j \neq n}} \big(\Gbf_{[2]}^{(j) \top} \Gbf_{[2]}^{(j)}\big)(\overline{r_{j} r_{j-1}}, \overline{k_{j} k_{j-1}}) \Big),
\end{aligned}
\end{equation}
which proves \eqref{eq:tr-sampling-results-joint-prob}.

\section{Detailed Complexity Analysis} \label{sec:detailed-complexity}

\subsection{CP-ALS-ES: Proposed Sampling Scheme for CP Decomposition} \label{sec:detailed-complexity-cp}

In this section we derive the computational complexity of the scheme proposed in Section~\ref{sec:sampling-for-cp}.

\paragraph{Computing $\Psibf \Abf^{\neq n}$}
First, we consider the costs of computing $\Psibf \Abf^{\neq n}$ as described in Section~\ref{sec:cp-efficient-subspace-embedding}:
\begin{itemize}
	\item Computing $\Ybf_j^{(0)}$ for all $j \in [2^q]$: 
	Each $\Cbf_j \Abf^{(\vbf(j))}$ costs at most $O(I_{\vbf(j)} R)$ to compute, and each $\Cbf_j (\ebf_1 \onebf_{1 \times R})$ costs $O(R)$ to compute.
	Since $2^q \leq 2N$, the total cost for this step is therefore $O(R \sum_{j \neq n} I_j)$.
	
	\item Computing $\Ybf_j^{(m)}$ for all $m \in [q]$ and all $j \in [2^{q-m}]$: 
	A single $J_1 \times J_1^2$ TensorSketch costs $O(R J_1 \log J_1)$ to apply to a matrix of the form $\Ybf_{2j-1}^{(m-1)} \odot \Ybf_{2j}^{(m-1)}$. 
	Such a TensorSketch is applied a total of $\sum_{m=1}^q 2^{q-m} = 2^q - 1 = O(N)$ times, so the total cost of this whole step is therefore $O(N R J_1 \log J_1)$.
\end{itemize}
The cost for computing $\Psibf \Abf^{\neq n}$ is therefore 
\begin{equation} \label{eq:cp-cost-1}
	O\bigg(R\Big(N J_1 \log J_1 + \sum_{j \neq n} I_j\Big)\bigg).
\end{equation}

\paragraph{Drawing $J_2$ Samples}
Second, we consider the cost of drawing $J_2$ samples in $[\prod_{j \neq n} I_j]$ from the distribution $\qbf$ as described in Section~\ref{sec:cp-sampling-scheme}:
\begin{itemize}
	\item One-time costs: 
	Computing the SVD of $\Psibf \Abf^{\neq n}$ costs $O(J_1 R^2)$.
	Computing $\Phibf = \Vbf_1 \Sigmabf_1^{-1} (\Vbf_1 \Sigmabf_1^{-1})^\top$ costs $O(R^3)$.
	Moreover, we can compute all products $\Abf^{(j)\top} \Abf^{(j)}$ for $j \neq n$ upfront for a cost of $O(R^2 \sum_{j \neq n} I_j)$.
	The sum of these one-time costs is $O(R^2(J_1 + R + \sum_{j \neq n} I_j))$.

	\item Cost of sampling $J_2$ indices:	
	Since each $\Abf^{(j)\top} \Abf^{(j)}$ for $j \neq n$ has already been computed, the cost of computing the probability $\Pb(i_m \mid (i_j)_{j<m, j \neq n})$ for a single set $(i_j)_{j \leq m, j \neq n}$ via \eqref{eq:cp-sampling-results-joint-prob} and \eqref{eq:conditional-probability-im} is $O(R^2 N)$.
	The total cost for computing the whole distribution for $i_m \in [I_m]$, for all $m \in [N] \setminus \{n\}$, is therefore $O(R^2 N \sum_{j \neq n} I_j)$.
	Since the main cost of sampling an index $i = \overline{i_1 \cdots i_{n-1} i_{n+1} \cdots i_N}$ is computing the distribution for each subindex, and we need to sample a total of $J_2$ samples, it follows that the total cost of drawing $J_2$ samples is $O(J_2 R^2 N \sum_{j \neq n} I_j)$.
\end{itemize}
In total, when including both one-time and per-sample costs, we get a cost for drawing $J_2$ samples from $\qbf$ of
\begin{equation} \label{eq:cp-cost-2}
	O \bigg( R^2 \Big( J_1 + R + J_2 N \sum_{j \neq n} I_j \Big) \bigg).
\end{equation}

\paragraph{Sampled Least Squares Problem}
Finally, we consider the cost of constructing and solving the sampled least squares problem once the $J_2$ samples in $[\prod_{j \neq n} I_j]$ have been drawn:
\begin{itemize}
	\item Once the $J_2$ samples in $[\prod_{j \neq n} I_j]$ are drawn, it costs $O(J_2 R N)$ to form $\Sbf \Abf^{\neq n}$, and $O(J_2 I_n)$ to form $\Sbf \Xbf_{(n)}^{\top}$.
	This can be done implicitly without forming the matrices $\Sbf$, $\Abf^{\neq n}$, and $\Xbf_{(n)}^\top$.
	
	\item The cost of computing the solution $\tilde{\Abf}^\top = (\Sbf \Abf^{\neq n})^\dagger \Sbf \Xbf_{(n)}^\top$ using a standard method (e.g., via QR decomposition) is $O(J_2 R^2 + J_2 R I_n)$; see Section~5.3.3 in \citet{golub2013MatrixComputations} for details.
\end{itemize}
In total, the costs of constructing and solving the least squares problem is therefore
\begin{equation} \label{eq:cp-cost-3}
	O(J_2 R(N + R + I_n)).
\end{equation}

\paragraph{Total Per-Iteration Cost for CP-ALS-ES} 
Recall that for each iteration of CP-ALS, we need to solve $N-1$ least squares problems.
Consequently, adding the costs in \eqref{eq:cp-cost-1}, \eqref{eq:cp-cost-2}, \eqref{eq:cp-cost-3} and multiplying by $N-1$, we get a total cost per iteration of 
\begin{equation} \label{eq:cp-full-complexity}
	O\bigg(R N^2 J_1 \log J_1 + R^2 N \Big(J_1 + R + J_2 N \sum_{j \neq n} I_j \Big) + J_2 R N I_n\bigg).
\end{equation}
If the sketch rates $J_1$ and $J_2$ are chosen according to \eqref{eq:cp-J1} and \eqref{eq:cp-J3}, this per-iteration cost becomes 
\begin{equation} \label{eq:cp-full-complexity-J-replaced}
	O\bigg( \frac{R^3 N^3}{\delta} \log\Big( \frac{R^2 N}{\delta} \Big) + \frac{R^4 N^2}{\delta} + \Big( R^3 N^2 \sum_{j \neq n} I_j + R^2 N I_n \Big) \max\Big( \log\Big(\frac{R}{\delta}\Big), \frac{1}{\varepsilon\delta} \Big)  \bigg).
\end{equation}

\subsection{TR-ALS-ES: Proposed Sampling Scheme for Tensor Ring Decomposition} \label{sec:detailed-complexity-tr}

In this section we derive the computational complexity of the scheme proposed in Section~\ref{sec:sampling-for-tr}.

\paragraph{Computing $\Psibf \Gbf_{[2]}^{\neq n}$}
First, we consider the computation of $\Psibf \Gbf_{[2]}^{\neq n}$ described in Section~\ref{sec:tr-efficient-subspace-embedding}:
\begin{itemize}
	\item Computing $\Ybf_j^{(0)}$ for all $j \in [2^q]$: 
	Recall that computing $\Cbf_j \Hbf^{(j)}$ costs $\nnz(\Hbf^{(j)})$.
	Consequently, the cost of computing all $\Ybf_j^{(0)}$ for $N \leq j \leq 2^q$ is just $O(1)$.
	The total cost for this step is therefore $O(\sum_{j = 1}^N I_j R_{j-1} R_j)$.
	
	\item Computing $\Ybf_j^{(m)}$ for all $m \in [q]$ and all $j \in [2^{q-m}]$: 
	Computing $\Tbf_j^{(m)}( \Ybf_{2j-1}^{(m-1)}(:, \overline{k_1 k_2}) \otimes \Ybf_{2j}^{(m-1)}(:, \overline{k_2 k_3}) )$ requires applying a $J_1 \times J_1^2$ TensorSketch to the Kronecker product of two vectors, which costs $O(J_1 \log J_1)$.
	This needs to be done for each $k_2 \in [K_{2j}^{(m-1)}]$ when computing the sum in \eqref{eq:Ymj}.
	This sum, in turn, needs to be computed for all $k_1 \in [K_{2j-1}^{(m-1)}]$, $k_3 \in [K_{2j+1}^{(m-1)}]$ and $j \in [2^{q-m}]$.
	Doing this for each $m \in [q]$ brings the total cost of this step to
	\begin{equation}
		O\Big( \sum_{m=1}^q \sum_{j=1}^{2^{q-m}} \sum_{k_1=1}^{K_{2j-1}^{(m-1)}} \sum_{k_2=1}^{K_{2j}^{(m-1)}} \sum_{k_3=1}^{K_{2j+1}^{(m-1)}} J_1 \log J_1 \Big).
	\end{equation}
	As we will see further down, this expression simplifies considerably if all $R_i$ are assumed to be equal.
\end{itemize}
Adding up the per-column costs above and multiplying them by the number of columns $R_{n-1} R_n$, we get that the cost for computing $\Psibf \Gbf_{[2]}^{\neq n}$ is 
\begin{equation} \label{eq:tr-cost-1}
	O\bigg(R_{n-1} R_n \Big(\sum_{j = 1}^N I_j R_{j-1} R_j + \sum_{m=1}^q \sum_{j=1}^{2^{q-m}} \sum_{k_1=1}^{K_{2j-1}^{(m-1)}} \sum_{k_2=1}^{K_{2j}^{(m-1)}} \sum_{k_3=1}^{K_{2j+1}^{(m-1)}} J_1 \log J_1 \Big)\bigg).
\end{equation}

\paragraph{Drawing $J_2$ Samples}
Second, we consider the cost of drawing $J_2$ samples in $[\prod_{j \neq n} I_j]$ from the distribution $\qbf$ as described in Section~\ref{sec:tr-sampling-scheme}:
\begin{itemize}
	\item One-time costs: 
	Computing the SVD of $\Psibf \Gbf_{[2]}^{\neq n}$ costs $O(J_1 (R_{n-1}R_n)^2)$.
	Computing $\Phibf = \Vbf_1 \Sigmabf_1^{-1} (\Vbf_1 \Sigmabf_1^{-1})^\top$ costs $O((R_{n-1}R_n)^3)$.
	Moreover, we can compute all products $\Gbf_{[2]}^{(j) \top} \Gbf_{[2]}^{(j)}$ for $j \neq n$ upfront for a cost of $O(\sum_{j \neq n} (R_{j-1} R_j)^2 I_j)$.
	The sum of these one-time costs is $O(J_1 (R_{n-1} R_n)^2 + (R_{n-1} R_n)^3 + \sum_{j \neq n} (R_{j-1} R_j)^2 I_j)$.
	
	\item Cost of sampling $J_2$ indices:
	The main cost of drawing the samples is computing the sampling distributions.
	Even though the number of terms in the sum of \eqref{eq:tr-sampling-results-joint-prob} is exponential in $N$, the joint probability distribution can be computed efficiently.
	We discuss how to do this in Remark~\ref{remark:computing-tr-probabilities}.
	The cost of doing this for one set of indices $(i_j)_{j \leq m, j \neq n}$ is given in \eqref{eq:tensor-evaluation-cost}. 
	Repeating this for all $i_j \in [I_j]$, which is required to get the distribution for the $j$th index, brings the cost to 
	\begin{equation} 
		O\Big( I_j R_N^2 \sum_{d=1}^{N-1} R_d^2 R_{d+1}^2 \Big).
	\end{equation}
	When this is repeated for all $N$ indices, and a total of $J_2$ times to get all samples, this brings the cost to 
	\begin{equation} 
		O\bigg( J_2 \Big( \sum_{j=1}^N I_j \Big) R_N^2 \sum_{d=1}^{N-1} R_d^2 R_{d+1}^2 \bigg).
	\end{equation}
\end{itemize}
Adding the one-time costs and the costs associated to computing the distributions, we get the following total cost for drawing $J_2$ samples:
\begin{equation} \label{eq:tr-cost-2}
	O\bigg( J_1 (R_{n-1} R_n)^2 + (R_{n-1} R_n)^3 + J_2 \Big(\sum_{j=1}^N I_j \Big) R_N^2 \sum_{d=1}^{N-1} R_d^2 R_{d+1}^2 \bigg).
\end{equation}

\paragraph{Sampled Least Squares Problem}
Finally, we consider the cost of constructing and solving the sampled least squares problem once the $J_2$ samples in $[\prod_{j \neq n} I_j]$ have been drawn:
\begin{itemize}
	\item Once $J_2$ samples in $[\prod_{j \neq n} I_j]$ are drawn, the sketched design matrix $\Sbf \Gbf_{[2]}^{\neq n}$ can be computed efficiently without having to form the full matrix $\Gbf_{[2]}^{\neq n}$. 
	We provide further details in Remark~\ref{remark:computing-sketched-tr-design-matrix}.
	With this approach, the cost of forming $\Sbf \Gbf_{[2]}^{\neq n}$ is
	\begin{equation}
		O\Big(J_2 R_n \sum_{j \in [N] \setminus \{n, n+1\}} R_{j-1} R_{j}\Big).
	\end{equation}
	Forming $\Sbf \Xbf_{[n]}^{\top}$ by sampling the appropriate rows costs $O(J_2 I_n)$.

	\item The cost of computing the solution $\tilde{\Gbf}^\top = (\Sbf \Gbf_{[2]}^{\neq n})^\dagger \Sbf \Xbf_{[n]}^\top$ using a standard method (e.g., via QR decomposition) is $O(J_2 (R_{n-1} R_n)^2 + J_2 R_{n-1} R_n I_n)$; see Section~5.3.3 in \citet{golub2013MatrixComputations} for details.
\end{itemize}
In total, the cost of constructing and solving the least squares problem is therefore
\begin{equation} \label{eq:tr-cost-3}
	O\bigg( J_2 \Big( R_n \sum_{j \in [N] \setminus \{ n, n+1 \}} R_{j-1} R_j + (R_{n-1} R_n)^2 + R_{n-1} R_n I_n \Big) \bigg).
\end{equation}

\paragraph{Total Per-Iteration Cost for TR-ALS-ES}
Recall that for each iteration of TR-ALS, we need to solve $N-1$ least squares problems.
Consequently, adding the costs in \eqref{eq:tr-cost-1}, \eqref{eq:tr-cost-2}, \eqref{eq:tr-cost-3} and multiplying by $N-1$, we get the total per-iteration cost.
If we assume that $R_j = R$ and $I_j = I$ for all $j \in [N]$, the expression simplifies considerably and we get a total per-iteration cost of
\begin{equation}
	O(N^2 R^5 J_1 \log J_1 + N^3 I R^6 J_2).
\end{equation}
If the sketch rates $J_1$ and $J_2$ are chosen according to \eqref{eq:tr-J1} and \eqref{eq:tr-J3}, this per-iteration cost becomes
\begin{equation}
	O\bigg( \frac{N^3 R^9}{\delta} \log\Big( \frac{N R^4}{\delta} \Big) + N^3 I R^8 \cdot \max\Big( \log\Big( \frac{R^2}{\delta} \Big), \frac{1}{\varepsilon \delta} \Big)\bigg).
\end{equation}

\begin{remark} \label{remark:computing-tr-probabilities}
	At first sight, the joint probability computation in \eqref{eq:tr-sampling-results-joint-prob} looks expensive since the number of terms in the sum is exponential in $N$.
	However, since not all summation indices $r_j$ and $k_j$ appear in every term, the summation can be done more efficiently.
	In fact, the computation \eqref{eq:tr-sampling-results-joint-prob} can be viewed as the evaluation of a tensor ring, which can be done efficiently by contracting core tensors pairwise.
	To see this, define core tensors $\Ce^{(j)}$ for $j \in [N]$ as follows:
	\begin{itemize}
		\item For $j \leq m$ and $j \neq n$, let $\Ce^{(j)} \in \Rb^{R_{j-1}^2 \times I_j \times R_{j}^2}$ be defined elementwise via
		\begin{equation}
			\Ce^{(j)}(\overline{r_{j-1} k_{j-1}}, i_j ,\overline{r_j k_j}) \defeq \Gbf_{[2]}^{(j)}(i_j, \overline{r_{j} r_{j-1}}) \Gbf_{[2]}^{(j)}(i_j, \overline{k_{j} k_{j-1}}).
		\end{equation}
		
		\item For $m < j \leq N$ and $j \neq n$, let $\Ce^{(j)} \in \Rb^{R_{j-1}^2 \times 1 \times R_j^2}$ be defined elementwise via
		\begin{equation}
			\Ce^{(j)}(\overline{r_{j-1} k_{j-1}}, 1,\overline{r_j k_j}) \defeq \big( \Gbf_{[2]}^{(j)\top} \Gbf_{[2]}^{(j)} \big)(\overline{r_{j} r_{j-1}}, \overline{k_{j} k_{j-1}}).
		\end{equation}
		
		\item For $j = n$, let $\Ce^{(j)} = \Ce^{(n)} \in \Rb^{R_{n-1}^2 \times 1 \times R_n^2}$ be defined elementwise via
		\begin{equation}
			\Ce^{(n)}(\overline{r_{n-1} k_{n-1}}, 1,\overline{r_n k_n}) \defeq \frac{1}{C} \Phibf (\overline{r_{n-1} r_n}, \overline{k_{n-1} k_n}).
		\end{equation}
	\end{itemize}
	We can now rewrite the expression in \eqref{eq:tr-sampling-results-joint-prob} as
	\begin{equation} \label{eq:joint-probability-as-contraction}
		\Pb((i_j)_{j \leq m, j \neq n}) = \TR(\Ce^{(1)}, \ldots, \Ce^{(N)})_{\xi_1, \ldots, \xi_{N}},
	\end{equation}
	where
	\begin{equation}
		\xi_j \defeq
		\begin{cases}
			i_j & \text{if } j \leq m, j \neq n, \\
			1 & \text{otherwise}.
		\end{cases}
	\end{equation}
	As discussed in \citet{zhao2016TensorRing}, the value of an entry in a tensor ring can be computed via a sequence of matrix-matrix products follows by taking the matrix trace:
	\begin{equation} \label{eq:trace-matrix-product}
		\TR(\Ce^{(1)}, \ldots, \Ce^{(N)})_{\xi_1, \ldots, \xi_{N}} = \trace\big( \Ce^{(1)}(:, \xi_1, :) \cdot \Ce^{(2)}(:, \xi_2, :) \cdots \Ce^{(N)}(:, \xi_N, :) \big),
	\end{equation}
	where each $\Ce^{(j)}(:, \xi_j, :)$ is treated as a $R_{j-1}^2 \times R_{j}^2$ matrix.
	If the matrix product in \eqref{eq:trace-matrix-product} is done left to right, evaluating the right hand side costs
	\begin{equation} \label{eq:tensor-evaluation-cost}
		O\Big( R_N^2 \sum_{j=1}^{N-1} R_j^2 R_{j+1}^2 \Big).
	\end{equation}
\end{remark}

\begin{remark} \label{remark:computing-sketched-tr-design-matrix}
	As described by \citet{malik2021SamplingBasedMethod}, it is possible to construct the sketched design matrix $\Sbf \Gbf_{[2]}^{\neq n}$ efficiently without first forming the full matrix $\Gbf_{[2]}^{\neq n}$.
	To see how, note that each row $\Gbf_{[2]}^{\neq n}(i, :)$ is the vectorization of the tensor slice $\Ge^{\neq n}(:, i, :)$ due to Definition~\ref{def:unfolding}.
	From Definition~\ref{def:subchain}, the tensor slice $\Ge^{\neq n}(:, i, :)$ is given by
	\begin{equation}
		\Ge^{\neq n}(:,\overline{i_{n+1} \cdots i_N i_1 \cdots i_{n-1}},:) = \Ge^{(n+1)}(:, i_{n+1}, :) \cdots \Ge^{(N)}(:, i_N, :) \cdot \Ge^{(1)}(:, i_1, :) \cdots \Ge^{(n-1)}(:, i_{n-1}, :).
	\end{equation}
	Suppose $\vbf \in [\prod_{j \neq n} I_j]^{J_2}$ contains the $J_2$ sampled indices corresponding to the sketch $\Sbf$.
	Let $\tilde{\Ge}^{\neq n} \in \Rb^{R_n \times J_2 \times R_{n-1}}$ be a tensor which we define as follows:
	For each $j \in [J_2]$, let $i = \overline{i_{n+1} \cdots i_N i_1 \cdots i_{n-1}} \defeq \vbf(j)$ and define
	\begin{equation} \label{eq:computing-sketched-slice}
		\tilde{\Ge}^{\neq n}(:,j,:) \defeq \frac{1}{\sqrt{J_2 \qbf(i)}} \Ge^{(n+1)}(:, i_{n+1}, :) \cdots \Ge^{(N)}(:, i_N, :) \cdot \Ge^{(1)}(:, i_1, :) \cdots \Ge^{(n-1)}(:, i_{n-1}, :).
	\end{equation}
	We now have $\Sbf \Gbf_{[2]}^{\neq n} = \tilde{\Gbf}_{[2]}^{\neq n}$.
	If the matrix product in \eqref{eq:computing-sketched-slice} is computed from left to right, it costs $O(R_n \sum_{j \in [N] \setminus \{n, n+1\}} R_{j-1} R_{j})$.
	Since this needs to be computed for each $j \in [J_2]$, the total cost for computing $\Sbf \Gbf_{[2]}^{\neq n}$ via this scheme is 
	\begin{equation}
		O\Big(J_2 R_n \sum_{j \in [N] \setminus \{n, n+1\}} R_{j-1} R_{j}\Big).
	\end{equation}
	We refer the reader to \citet{malik2021SamplingBasedMethod} for further details.
\end{remark}

\subsection{Complexity Analysis of Competing Methods} \label{sec:detailed-complexity-other}

In this section we provide a few notes on how we computed the computational complexity of the other methods we compare with in Tables~\ref{tab:cp-complexity-comparison} and \ref{tab:tr-complexity-comparison}.

\subsubsection{CP-ALS}

The standard way to implement CP-ALS is given in Figure~3.3 in \citet{kolda2009TensorDecompositions}.
The leading order cost per least squares solve for that algorithm is
\begin{equation}
	O(N I R^2 + R^3 + N I^{N-1} R + I^N R).
\end{equation}
Since $N$ such least squares problems need to be solved each iteration, the per-iteration cost is
\begin{equation}
	O(N^2 I R^2 + N R^3 + N^2 I^{N-1} R + N I^N R).
\end{equation}

When $N$ is large, this becomes $O(N (N+I) I^{N-1} R)$ which is what we report in Table~\ref{tab:cp-complexity-comparison}.

\subsubsection{SPALS}

\citet{cheng2016SPALSFast} only give the sampling complexity for the case when $N=3$ in their paper.
For arbitrary $N$, and without any assumptions on the rank of the factor matrices or the Khatri--Rao product design matrix, their scheme requires $J \gtrsim R^N \log(I_n / \delta) / \varepsilon^2$ samples when solving for the $n$th factor matrix in order to achieve the additive error guarantees in Theorem~4.1 of their paper.%
\footnote{
	If the Khatri--Rao product design matrix is full rank, which happens if all factor matrices are full rank, then $J \gtrsim R^{N-1} \log(I_n / \delta) / \varepsilon^2$ samples will suffice.
}

SPALS requires a one-time upfront cost of $\nnz(\Xe)$ in order to compute the second term in Equation~(5) in \citet{cheng2016SPALSFast}. 
In SPALS, the $n$th factor is updated via
\begin{equation}
	\Abf^{(n)} = \Xbf_{(n)} \Sbf^\top \Sbf \Big( \kr_{\substack{j=N \\ j \neq n}}^{1} \Abf^{(j)} \Big) \Big( \startimes_{\substack{j=1 \\ j \neq n}}^N \Abf^{(j) \top} \Abf^{(j)} \Big)^{-1},
\end{equation}
where $\Sbf$ is a sampling matrix and $\circledast$ denotes elementwise (Hadamard) product.
When this is computed in the appropriate order, and if $\log$ factors are ignored and we assume that $I_n = I$ for all $n \in [N]$, then the cost of computing $\Abf^{(n)}$ is 
\begin{equation}
	\tilde{O}( N I R^2 + (N + I) R^{N+1} / \varepsilon^2 ).
\end{equation}
Notice that the cost of computing the sampling distribution is dominated by the cost above.
Since $N$ factor matrices need to be updated per iteration, the total per-iteration cost is
\begin{equation}
	\tilde{O} (N^2 I R^2 + N (N+I) R^{N+1} / \varepsilon^2).
\end{equation} 
When $N$ is large, this becomes $\tilde{O} ( N (N+I) R^{N+1} / \varepsilon^2 )$, which is what we report in Table~\ref{tab:cp-complexity-comparison}.

\subsubsection{CP-ARLS-LEV}

From Theorem~8 in \citet{larsen2020PracticalLeverageBased}, the sampling complexity for CP-ARLS-LEV required to achieve relative error guarantees when solving for the $n$th factor matrix is $J \gtrsim R^{N-1} \max(\log(R/\delta), 1/(\delta \varepsilon))$.
Solving the sampled least squares problem, which has a design matrix of size $J \times R$ and $I_n$ right hand sides via e.g.\ QR decomposition (see Section~5.3.3 in \citet{golub2013MatrixComputations}) will therefore cost $O((R + I_n) R^N \max(\log(R/\delta), 1/(\delta \varepsilon)))$.
Each iteration requires solving $N$ such least squares problems.
If we assume that $I_n = I$ for all $n \in [N]$ and ignore $\log$ factors, the per-iteration cost becomes
\begin{equation}
	\tilde{O} \big( N ( R + I ) R^N / (\delta \varepsilon) \big),
\end{equation}
which is what we report in Table~\ref{tab:cp-complexity-comparison}.

Consider the least squares problem in \eqref{eq:cp-optimization-An-matrix} with the design matrix $\Abf^{\neq n}$ defined in as in \eqref{eq:cp-design-matrix}.
When the leverage score sampling distribution is estimated as in CP-ARLS-LEV, the exponential dependence on $N$ in the sampling complexity cannot be improved.
The following example provides a concrete example when the exponential dependence is required.
\begin{example}
	Without loss of generality, consider the case $n = N$ in which case the least squares design matrix in \eqref{eq:cp-design-matrix} is
	\begin{equation}
		\Abf^{\neq N} = \Abf^{(N-1)} \odot \cdots \odot \Abf^{(1)}.
	\end{equation}
	Suppose all $\Abf^{(j)}$ for $j \in [J-1]$ are of size $R \times R$ and defined as
	\begin{equation}
		\Abf^{(j)} \defeq 
		\begin{bmatrix}
			1 		& \zerobf \\
			\zerobf & \Omegabf^{(j)}
		\end{bmatrix}
	\end{equation}
	where each $\Omegabf^{(j)} \in \Rb^{(R-1) \times (R-1)}$ has i.i.d.\ standard Gaussian entries.
	We assume the matrices $\Omegabf^{(j)}$ are all full-rank, which is true almost surely.
	The first column and row of $\Abf^{\neq N}$ are $\ebf_1$ and $\ebf_1^\top$, respectively.
	For $r \geq 2$ we have
	\begin{equation}
		\Abf^{\neq N}(:,r) = 
		\begin{bmatrix}
			0 \\
			\Omegabf^{(N-1)}(:,r)
		\end{bmatrix}
		\otimes \cdots \otimes
		\begin{bmatrix}
			0 \\
			\Omegabf^{(1)}(:,r)
		\end{bmatrix}.
	\end{equation}
	Since all the matrices $\Omegabf^{(j)}$ are full-rank it follows that their Kronecker product is full-rank (this follows from Theorem~4.2.15 in \citet{horn1994TopicsMatrix}).
	Since the columns $\Abf^{\neq N}(:,r)$ for $r \geq 2$ are equal to columns of $\Omegabf^{(N-1)} \otimes \cdots \otimes \Omegabf^{(1)}$ with some added zeros, it follows that they are linearly independent, and therefore the submatrix $\Gammabf \defeq (\Abf^{\neq N}(i,j))_{i \geq 2, j \geq 2}$ is full-rank.
	We may write 
	\begin{equation}
		\Abf^{\neq N} = 
		\begin{bmatrix}
			1 		& \zerobf \\
			\zerobf & \Gammabf
		\end{bmatrix} \in \Rb^{R^{N-1} \times R}.
	\end{equation}
	If a sampling matrix $\Sbf$ does not sample the first row of $\Abf^{\neq N}$, then $\Sbf \Abf^{\neq N}$ will be rank-deficient and relative error guarantees therefore unachievable.
	Since all $\Abf^{(j)}$ are square and full-rank, the sampling procedure used in CP-ARLS-LEV will sample rows of $\Abf^{\neq N}$ uniformly.
	In order to sample the first row of $\Abf^{\neq N}$ with probability at least 0.5 with uniform sampling, we clearly need to sample at least half of all rows of $\Abf^{\neq N}$, i.e., we need $J \geq R^{N-1}/2$.
\end{example}

\subsubsection{Methods for Tensor Ring Decomposition}

The complexities we report in Table~\ref{tab:tr-complexity-comparison} for other methods where taken directly from Table~1 in \citet{malik2021SamplingBasedMethod}.

\section{Additional Experiment Details} \label{sec:additional-experiments}

\subsection{Details on Algorithm Implementations}

Our implementation of CP-ARLS-LEV is based on Algorithm~3 in \citet{larsen2020PracticalLeverageBased}.
We do not use any hybrid-deterministic sampling, but we do combine repeated rows.
Some key functionality required for our CP-ALS-ES is written in C and incorporated into Matlab via the MEX interface.
Our own TR-ALS-ES is implemented by appropriately modifying the Matlab code for TR-ALS-Sampled by \citet{malik2021SamplingBasedMethod}.

\subsection{Datasets}

The photo used for the sampling distribution comparison was taken by Sebastian M\"{u}ller on Unsplash and is available at \url{https://unsplash.com/photos/l54ZALpH2_I}.
We converted this figure to gray scale by averaging the three color channels. 
We also cropped the image slightly to make the width and height a power of 2.
The tensorization is done following the ideas for visual data tensorization discussed in \citet{yuan2019HighorderTensor}.
Please see our code for precise details.

The COIL-100 dataset was created by \citet{nene1996ColumbiaObject} and is available for download at \url{https://www.cs.columbia.edu/CAVE/software/softlib/coil-100.php}.

\subsection{Sampling Distribution Plots and Computational Time}

We have included figures below that compare the sampling distributions used by our methods with those used by the previous state-of-the-art methods in the least squares problem considered in the first experiment in Section~\ref{sec:experiments}.
For a rank-10 CP decomposition of the tabby cat tensor, Figure~\ref{fig:CP-rank-10-probability-comparison} shows the exact leverage score distribution ($\pbf$ in Definition~\ref{def:leverage-score-sampling}), the sampling distribution used by CP-ARLS-LEV, and a realization (for $J_1 = 1000$) of the distribution our CP-ALS-ES uses.
Figure~\ref{fig:CP-rank-20-probability-comparison} shows the same things as Figure~\ref{fig:CP-rank-10-probability-comparison}, but for a rank-20 CP decomposition. 
For a rank-$(3,\ldots,3)$ tensor ring decomposition of the tabby cat tensor, Figure~\ref{fig:TR-rank-3-probability-comparison} shows the exact leverage score distribution, the sampling distribution used by TR-ALS-Sampled, and a realization (for $J_1=1000$) of the distribution our TR-ALS-ES uses.
Figure~\ref{fig:TR-rank-5-probability-comparison} shows the same things as Figure~\ref{fig:TR-rank-3-probability-comparison}, but for a rank-$(5,\ldots,5)$ tensor ring decomposition. 

Notice that the sampling distribution that our methods use follow the exact leverage score sampling distribution closely. 
The distributions used by CP-ARLS-LEV and TR-ALS-Sampled are less accurate.
In particular, when $R > I = 16$ for the CP decomposition (Figure~\ref{fig:CP-rank-20-probability-comparison}) or when $R_{n-1}R_n > I = 16$ for the tensor ring decomposition (Figure~\ref{fig:TR-rank-5-probability-comparison}), CP-ARLS-LEV and TR-ALS-Sampled sample from \emph{a uniform distribution}.
This is not an anomaly, but rather a direct consequence from how those methods estimate the leverage scores.
Our proposed methods, by contrast, handle those cases well.

\begin{figure}[h]
	\centering  
	\includegraphics[width=1\columnwidth]{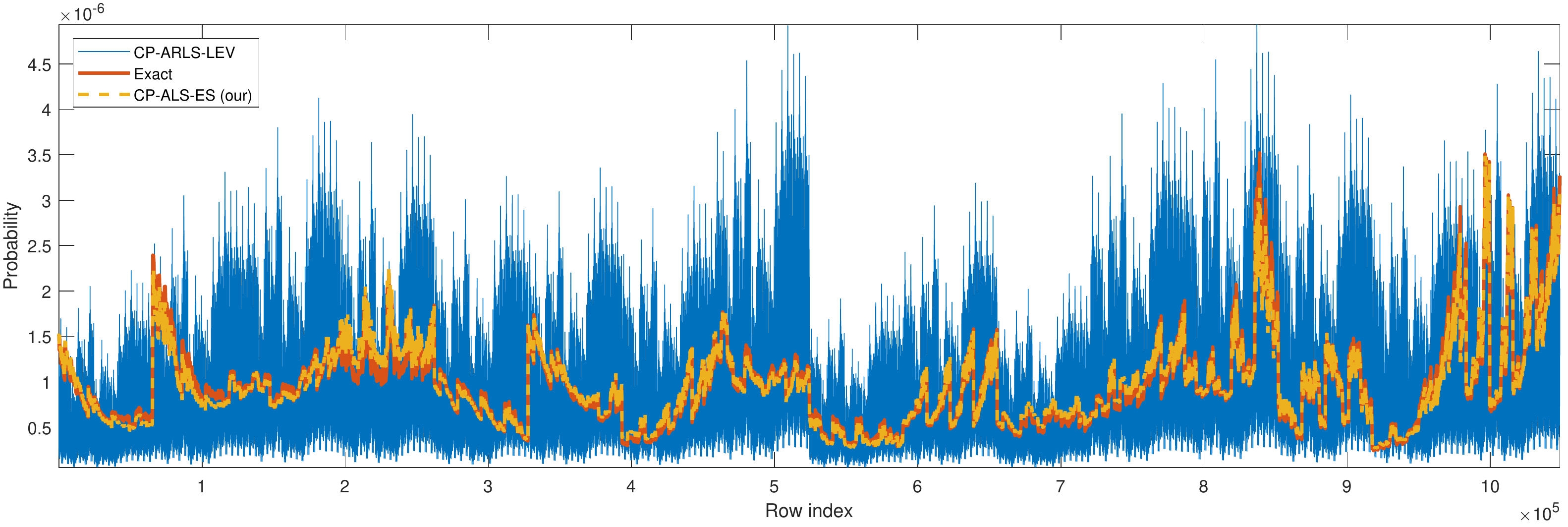}
	\caption{
		Comparison of the exact leverage score distribution, the sampling distribution used by CP-ARLS-LEV, and a realization (for $J_1=1000$) of the distribution used by our CP-ALS-ES.
		The least squares problem corresponds to solving for the 6th factor matrix in a rank-10 CP decomposition of the 6-way tabby cat tensor.
	}
	\label{fig:CP-rank-10-probability-comparison}
\end{figure}

\begin{figure}[h]
	\centering  
	\includegraphics[width=1\columnwidth]{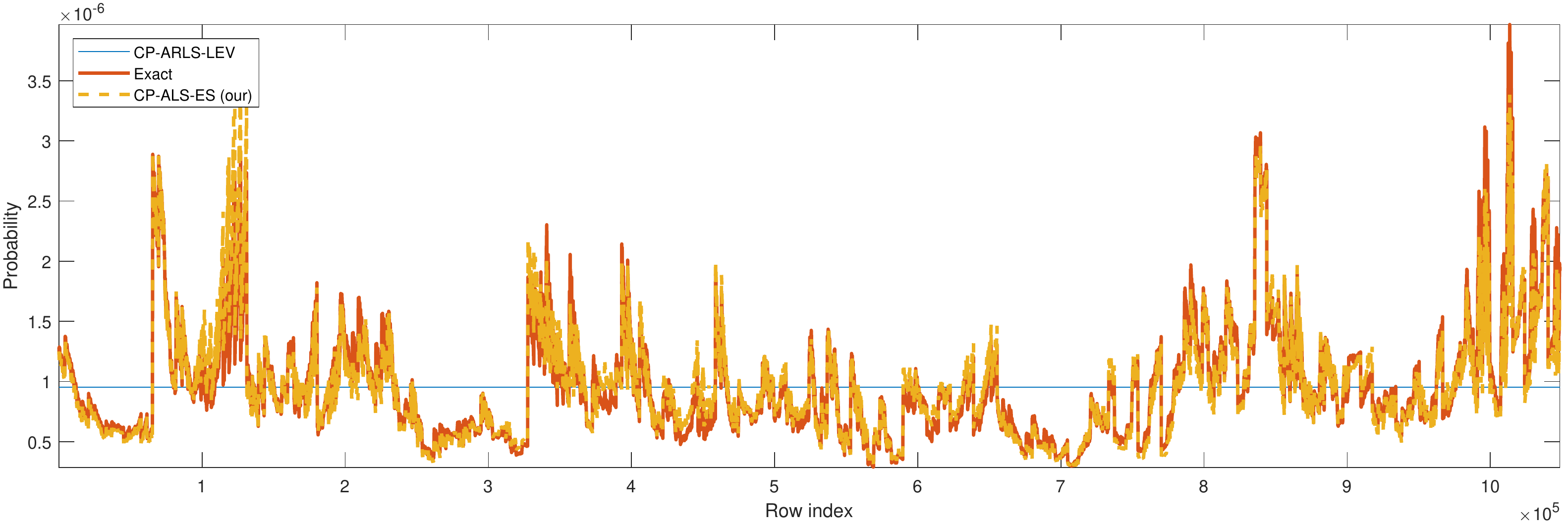}
	\caption{
		Same as Figure~\ref{fig:CP-rank-10-probability-comparison}, but for a rank-20 decomposition.
	}
	\label{fig:CP-rank-20-probability-comparison}
\end{figure}

\begin{figure}[h]
	\centering  
	\includegraphics[width=1\columnwidth]{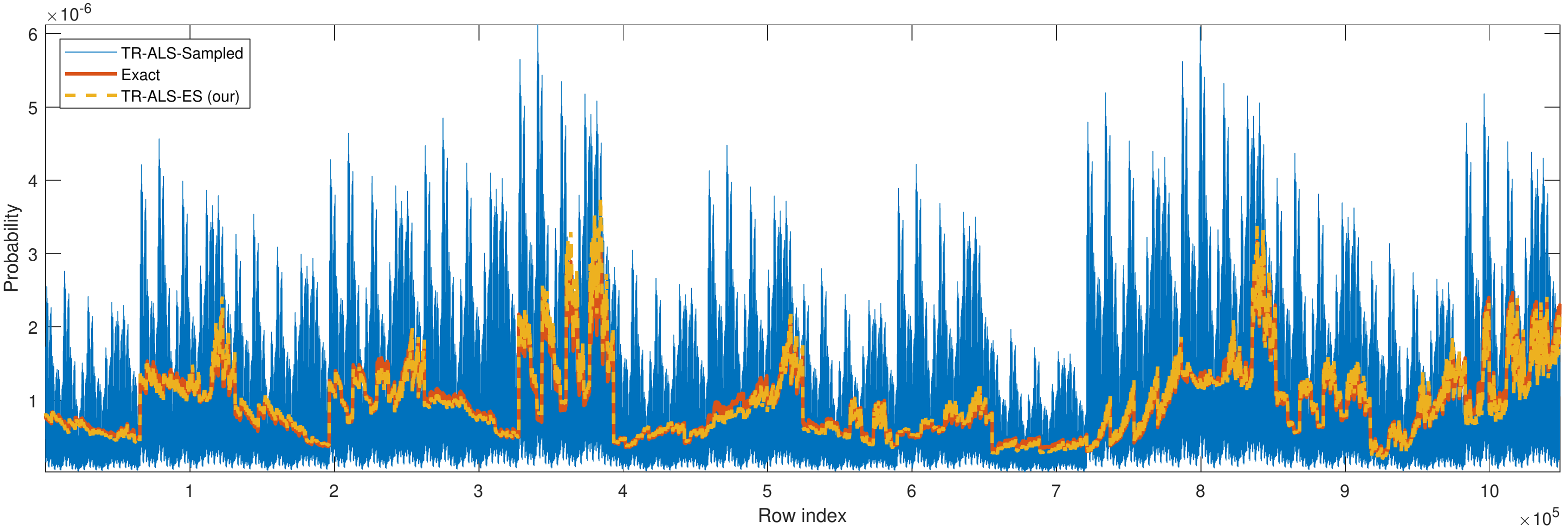}
	\caption{
		Comparison of the exact leverage score distribution, the sampling distribution used by TR-ALS-Sampled, and a realization (for $J_1 = 1000$) of the distribution used by our TR-ALS-ES.
		The least squares problem corresponds to solving for the 6th core tensor in a rank-$(3,\ldots,3)$ tensor ring decomposition of the 6-way tabby cat tensor.
	}
	\label{fig:TR-rank-3-probability-comparison}
\end{figure}

\begin{figure}[h]
	\centering  
	\includegraphics[width=1\columnwidth]{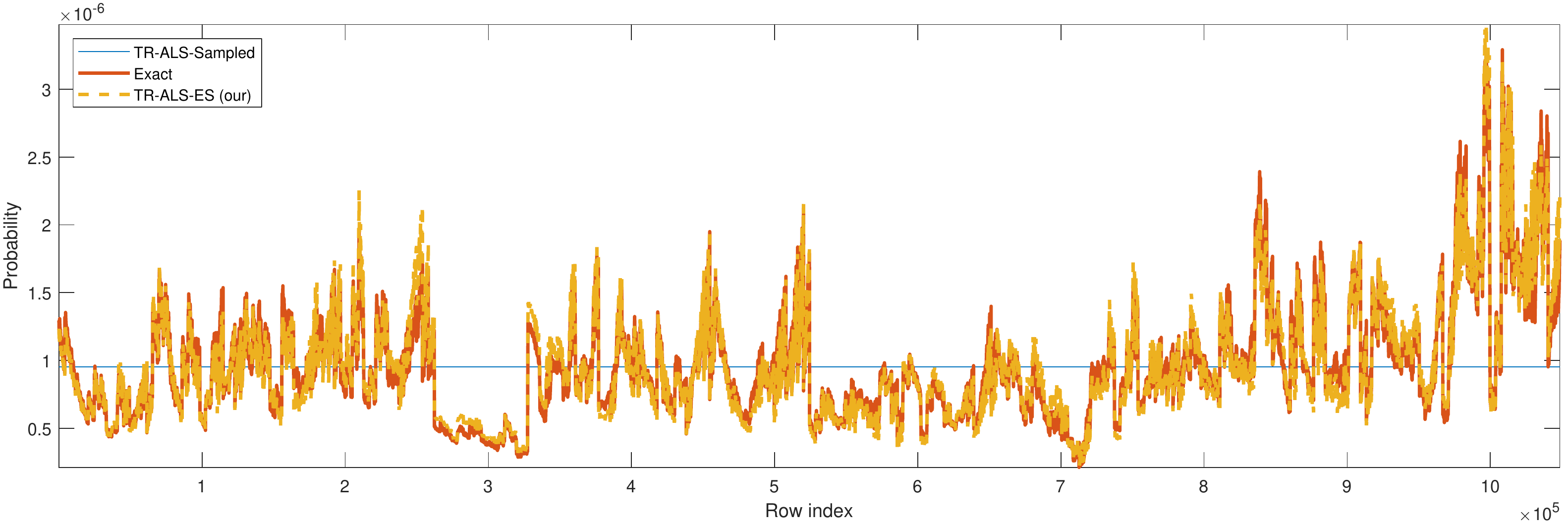}
	\caption{
		Same as Figure~\ref{fig:TR-rank-3-probability-comparison}, but for a rank-$(5,\ldots,5)$ decomposition.
	}
	\label{fig:TR-rank-5-probability-comparison}
\end{figure}

Tables~\ref{tab:KL-div-CP-time} and \ref{tab:KL-div-TR-time} report the time it took to compute the distributions used in Tables~\ref{tab:KL-div-CP} and \ref{tab:KL-div-TR}, respectively. 
Note that the different methods do not compute the full distributions the way we do in Tables~\ref{tab:KL-div-CP}--\ref{tab:KL-div-TR} and Figures~\ref{fig:CP-rank-10-probability-comparison}--\ref{fig:TR-rank-5-probability-comparison}, so these numbers are not representative of actual decomposition time and are only added here for completeness.

\begin{table}[ht!]
	\centering
	\caption{
		Time in seconds it took to compute the distributions used in Table~\ref{tab:KL-div-CP}.
		\label{tab:KL-div-CP-time}
	}
	\begin{tabular}{lll}  
		\toprule
		Method & $R=10$ & $R=20$ \\
		\midrule
		CP-ARLS-LEV 						& 0.01 & 0.01 \\
		CP-ALS-ES ($J_1 = \text{1e+4}$)		& 0.06 & 0.12 \\
		CP-ALS-ES ($J_1 = \text{1e+3}$) 	& 0.04 & 0.07 \\
		CP-ALS-ES ($J_1 = \text{1e+2}$) 	& 0.03 & 0.07 \\
		\bottomrule
	\end{tabular}
\end{table}

\begin{table}[ht!]
	\centering
	\caption{
		Time in seconds it took to compute the distributions used in Table~\ref{tab:KL-div-TR}.
		\label{tab:KL-div-TR-time}
	}
	\begin{tabular}{lll}  
		\toprule
		Method & $R=3$ & $R=5$ \\
		\midrule
		TR-ALS-Sampled 						& 0.01 & 0.01 \\
		TR-ALS-ES ($J_1 = \text{1e+4}$) 	& 0.07 & 0.21 \\
		TR-ALS-ES ($J_1 = \text{1e+3}$) 	& 0.03 & 0.10 \\
		TR-ALS-ES ($J_1 = \text{1e+2}$) 	& 0.03 & 0.10 \\
		\bottomrule
	\end{tabular}
\end{table}

\subsection{Feature Extraction Experiments} \label{sec:supp-feature-extraction-experiments}

We provide some further details on the feature extraction experiments in Section~\ref{sec:feature-extraction} in this section.
For a rank-25 CP decomposition, the 4th factor matrix is of size $7200 \times 25$.
We directly use this factor matrix as the feature matrix we feed to the $k$-NN method in Matlab.
For the rank-$(5, \ldots, 5)$ tensor ring decomposition, the 4th core tensor is of size $5 \times 7200 \times 5$. 
We turn this into a $7200 \times 25$ matrix via a classical mode-2 unfolding which we then use as the feature matrix in the $k$-NN algorithm.

In our feature extraction experiment we assume that both the labeled and unlabeled images are available when the tensor decomposition is computed.
This is a limitation since we might want to classify new unlabeled images that arrive after the decomposition has been computed without having to recompute the decomposition.
We now propose a potential approach to circumventing this limitation.
Adding a new image corresponds to adding new rows to $\Xe_{(4)}$ and $\Xe_{[4]}$.
The factor matrix $\Abf^{(4)}$ for the CP decomposition of this augmented tensor will have an additional row, while the number of rows will remain the same in the other factor matrices.
Similarly, the core tensor $\Ge^{(4)}$ for the tensor ring decomposition will have an additional lateral slice, while the number of lateral slices will remain the same for the other cores.
For the CP decomposition, a feature vector for the new image can therefore be computed via
\begin{equation} \label{eq:CP-feat-ext-new-samp}
	\abf^* = \argmin_{\abf} \| \Abf^{\neq 4} \abf^\top - \xbf^\top \|_2,
\end{equation}
where $\abf^* \in \Rb^{1 \times R}$ is the new row in $\Abf^{(4)}$ and $\xbf \in \Rb^{1 \times 49152}$ is the new row in $\Xe_{(4)}$.
For the tensor train decomposition, a feature vector for the new image can similarly be computed via
\begin{equation} \label{eq:TR-feat-ext-new-samp}
	\gbf^* = \argmin_{\gbf} \| \Gbf_{[2]}^{\neq 4} \gbf^\top - \hat{\xbf}^\top \|_2,
\end{equation} 
where $\gbf^* \in \Rb^{1 \times R_3 R_4}$ is a reshaped version of the new lateral slice in $\Ge^{(4)}$ and $\hat{\xbf} \in \Rb^{1 \times 49152}$ is the new row in $\Xe_{[4]}$.
The sampling techniques in Section~\ref{sec:efficient-sampling-for-TD} can be used to compute approximate solutions to \eqref{eq:CP-feat-ext-new-samp} and \eqref{eq:TR-feat-ext-new-samp} efficiently.
The feature vectors $\abf^*$ and $\gbf^*$ can now be used to classify the new image.

\subsection{Demonstration of Improved Complexity for the Tensor Ring Decomposition} \label{sec:demonstration-improvement-TR}

We construct a synthetic 10-way tensor that demonstrates the improved sampling and computational complexity of our proposed TR-ALS-ES over TR-ALS-Sampled.
It is constructed via \eqref{eq:tr-decomposition} from core tensors $\Ge^{(n)} \in \Rb^{3 \times 6 \times 3}$ for $n \in [10]$ with $\Ge^{(n)}(1,1,1) = 3$ and all other entries zero.
Additionally, i.i.d.\ Gaussian noise with standard deviation $0.01$ is added to all entries of the tensor.
Both methods are run for 20 iterations with target ranks $(3, 3, \ldots, 3)$ and are initialized using a variant of the randomized range finding approach proposed by \citet{larsen2020PracticalLeverageBased}, Appendix~F, adapted to the tensor ring decomposition.
TR-ALS-Sampled fails even when as many as \emph{half} (i.e., $J = 6^9/2 \approx 5.0\text{e+6}$) of all rows are sampled, taking 966 seconds.
By contrast, our TR-ALS-ES only requires a recursive sketch size of $J_1 = 1\text{e+4}$ and $J_2 = 1\text{e+3}$ samples to get an accurate solution, taking 41 seconds.
Our method improves the sampling complexity and compute time by 3 and 1 orders of magnitude, respectively.

\subsection{Preliminary Results From Experiments on the Tensor Train Decomposition} \label{sec:tt-experiments}

In this section we provide some preliminary results from experiments on the tensor train (TT) decomposition.
We do these experiments by running the different tensor ring decomposition algorithms with $R_0 = R_N = 1$ which makes the resulting decomposition a TT.
We refer to the methods by the same names as the tensor ring decomposition methods, but with ``TR'' replaced by ``TT.''

\begin{table}[ht!]
	\centering
	\caption{
		KL-divergence (lower is better) of the approximated sampling distribution from the exact one for a TT-ALS least squares problem with target TT-ranks $R_n = R$ for $1 \leq n \leq 5$.
		The TT-ALS least squares problem is identical to the TR-ALS problem in \eqref{eq:tr-optimization-Gn-matrix} but with the restriction $R_0=R_N=1$.
		\label{tab:KL-div-TT}
	}
	\begin{tabular}{lll}  
		\toprule
		Method & $R=3$ & $R=5$ \\
		\midrule
		TT-ALS-Sampled 						& 0.4843 & 0.2469 \\
		TT-ALS-ES ($J_1 = \text{1e+4}$) 	& 0.0004 & 0.0007 \\
		TT-ALS-ES ($J_1 = \text{1e+3}$) 	& 0.0087 & 0.0065 \\
		TT-ALS-ES ($J_1 = \text{1e+2}$) 	& 0.0425 & 0.0845 \\
		\bottomrule
	\end{tabular}
\end{table}

\begin{table}[ht!]
	\centering
	\caption{
		Run time, decomposition error and classification accuracy when using the TT decomposition for feature extraction.
		\label{tab:classification-results-TT}
	}
	\begin{tabular}{lrrr}  
		\toprule
		Method    						& Time (s) & Error & Accuracy (\%) \\
		\midrule
		TT-ALS							& 9440.1 & 0.44 & 95.2 \\
		TT-ALS-Sampled					&   12.3 & 0.44 & 94.1 \\
		\textbf{TT-ALS-ES (our)}	   	&   31.4 & 0.44 & 94.2 \\
		\bottomrule
	\end{tabular}
\end{table}

First, we repeat the experiments in Section~\ref{sec:sampling-ditribution-comparison} for the TT decomposition.
All settings are the same as for the tensor ring decomposition except that $R_0 = R_6 = 1$.
The results are shown in Table~\ref{tab:KL-div-TT}.
The discrepancy between the approximate leverage score sampling distribution that TT-ALS-Sampled samples from  and the exact one is greater than it is for the tensor ring decomposition (compare with Table~\ref{tab:KL-div-TR}).
The discrepancy of TT-ALS-ES is similar to that of TR-ALS-ES for $J_1 = 1\text{e+4}$ and $J_1 = 1\text{e+3}$ and smaller for $J_1 = 1\text{e+2}$. 
Since we are solving for the 6th core tensor out of 6, the theoretical bound on $J_1$ in \eqref{eq:tr-J1} becomes $J_1 \gtrsim N (R_5 R_6)^2/\delta = N R^2/\delta$ since $R_6=1$.
For the tensor ring decomposition with all $R_n=R$, the same bound is $J_1 \gtrsim N R^4 / \delta$.
A smaller discrepancy is therefore expected for TT-ALS-ES than for TR-ALS-ES.

Next, we repeat the feature extraction experiment in Section~\ref{sec:feature-extraction} for the TT decomposition.
In addition to the restriction $R_0=R_1=1$ we also increase the number of samples from 1000 to 3000 for both TT-ALS-Sampled and TT-ALS-ES since using fewer than 3000 samples yields poor results for both methods.
We also add a small Tikhonov regularization term (with regularization constant $10^{-2}$) in all least squares solves for both randomized methods in order to avoid numerical issues that otherwise appear for both. 
The TT-ranks are $R_n = 5$ for $1 \leq n \leq 3$.
The results are reported in Table~\ref{tab:classification-results-TT}.
The run time for TT-ALS is slightly faster than that of TR-ALS which is expected since the TT has fewer parameters than the tensor ring (compare with Table~\ref{tab:classification-results}).
The two randomized algorithms are a bit slower than they are for the tensor ring decomposition due to the larger number of samples being drawn.
The decomposition error is higher and the classification accuracy lower than they are for the tensor ring decomposition.
This is also expected since the TT decomposition has fewer parameters.

\end{document}